\documentclass{amsart}

\usepackage{xcolor}
\usepackage{mathtools}
\usepackage{amsmath}
\usepackage{amssymb}
\usepackage{amsthm}
\usepackage{enumitem}
\usepackage[colorlinks, citecolor=blue, linkcolor=red]{hyperref}
\usepackage[noabbrev,capitalize]{cleveref}

\newtheorem{theorem}{Theorem}[section]
\newtheorem{proposition}[theorem]{Proposition}
\newtheorem{lemma}[theorem]{Lemma}
\newtheorem{corollary}[theorem]{Corollary}

\newtheorem{meta}[theorem]{\textbf{Meta-Theorem}}
\newtheorem{cor}[theorem]{\textbf{Corollary}}
\theoremstyle{remark}
\newtheorem{rem}[theorem]{\textbf{Remark}}
\newtheorem{exe}[theorem]{\textbf{Example}}

\newcommand{\KN}{\mathbin{\bigcirc\mspace{-15mu}\wedge\mspace{3mu}}}

\def\sideremark#1{\ifvmode\leavevmode\fi\vadjust{\vbox to0pt{\vss
			\hbox to 0pt{\hskip\hsize\hskip1em
				\vbox{\hsize3cm\tiny\raggedright\pretolerance10000
					\noindent #1\hfill}\hss}\vbox to8pt{\vfil}\vss}}}

\numberwithin{equation}{section}

\begin{document}
	
	\title[Four Dimensional Hypersurfaces in space forms]{Topological and rigidity results for four-dimensional hypersurfaces in space forms}
	
	
	\maketitle

	\begin{center}
		\textsc{\textmd{ Davide Dameno \footnote{Dipartimento di Matematica ``Giuseppe Peano", Universit\`{a} degli Studi di Torino, Via Carlo Alberto 10, 10123 Torino, Italy.
					E-mail: davide.dameno@unito.it.}, Aaron J. Tyrrell \footnote{Department of Mathematics, University of Notre Dame, Notre Dame, IN 46556, USA. E-mail: atyrrell@nd.edu}.}}
	\end{center}
	
	\begin{abstract}
		Exploiting the special features of four-dimensional Riemannian geometry, we derive topological and rigidity results for hypersurfaces immersed in space forms of dimension 5. First, we provide a complete description of the Weyl tensor for four-dimensional hypersurfaces, by means of which we derive a new characterization result for isoparametric hypersurfaces; then, we prove sharp topological bounds on the Weyl functional for closed, minimal hypersurfaces, involving the Euler characteristic in the case of an ambient space with constant non-negative sectional curvature. Then, inspired by a famous conjecture by Chern and the so-called second pinching problem, we find estimates for the norm of the second fundamental form in terms of the Euler characteristic in the minimal, constant scalar curvature case, under a cross-sectional area assumption. Finally, we prove some rigidity results by means of integral inequalities on the derivatives of the second fundamental form, also dealing with special curvature conditions, such as half harmonic Weyl curvature and Bach-flatness. We also extend some of the local results to the case of a locally conformally flat 5-dimensional ambient space.
		
	\end{abstract}
		
		\maketitle

		\section{Introduction and main results}
		\label{sec:intro}
		Submanifold theory has been a pivotal research topic in many branches of mathematics and physics for several decades, as a natural extension of the theory of surfaces: in particular, the theory of hypersurfaces has been investigated and exploited in countless areas of differential geometry and geometric analysis. 
		
		In this paper, we are interested in isometrically immersed hypersurfaces carrying additional structures, both from a geometric and topological viewpoint. Over the past century, many authors have addressed the problem of characterizing and classifying immersed hypersurfaces whose induced metrics are endowed with special properties: for instance, extensive work has been done on hypersurfaces equipped with locally conformally flat metrics \cite{docarmodaj, docarmomercuri, NiskikawaMaeda1974} and Einstein metrics \cite{Fialkow1938, ryan} or their generalizations \cite{lawson}. 
		
		A fundamental subclass of hypersurfaces is given by the \emph{minimal} ones, i.e. critical points of the area functional for compactly supported variations, or, equivalently, hypersurfaces whose mean curvature $H$ identically vanishes. A famous problem regarding immersed minimal hypersurfaces is the so-called \emph{Chern conjecture} (we refer the reader to the useful surveys \cite{ge2010chern, scherfner2012review}): in its first formulation \cite{chern, cherndocarmokob}, motivated by a celebrated result due to Simons \cite{simons}, the conjecture asserts that, given a closed, minimal hypersurface $M^n$ isometrically immersed in the standard sphere $\mathbb{S}^{n+1}$, if the induced metric on $M$ has constant scalar curvature $R$, then the set of possible values of $R$ is discrete. As a simple consequence of Gauss equations, if the ambient space is a space form, the scalar curvature is constant if and only if the squared norm of the second fundamental form is constant. 
		
		The conjecture can be stated in a stronger version: indeed, up to now, the only known examples of minimal hypersurfaces with constant scalar curvature in spheres are \emph{isoparametric}, i.e their principal curvatures are constant functions. The theory of isoparametric hypersurfaces was first developed by Cartan \cite{cartan38, cartan, cartan40}: later, M\"{u}nzner \cite{munzner} showed that, if the ambient space is the standard sphere, then the number $m$ of distinct principal curvatures can only be $1,2,3,4$ or $6$, which implies that, if the hypersurface is minimal, $R=n(n-m)$ (see also
		\cite{pengterng} for further details).
		In the minimal case, for instance, if $m=2$, we obtain the so-called \emph{Clifford hypersurfaces}
		\[
		\mathbb{S}^k\left(\sqrt{\dfrac{k}{n}}\right)\times \mathbb{S}^{n-k}\left(\sqrt{\dfrac{n-k}{n}}\right),
		\]
		while, for $m=3$, we find \emph{Cartan's minimal hypersurfaces};
		for the sake of completeness, we refer the reader to the beautiful surveys 
		\cite{cecilryan, chi} and the references therein for a detailed discussion about the classification of isoparametric hypersurfaces in spheres.  
		With this in mind, the strong Chern conjecture can therefore be stated as follows: every closed, minimal hypersurface with constant scalar curvature in the standard sphere must be isoparametric. 
		Up to now, the strong version of the conjecture (and, therefore, the weak one) was fully proven for minimal hypersurfaces in $\mathbb{S}^4$ by Peng, Terng \cite{pengterngmathann} and Chang \cite{changJDG}, who later extended the result to the \emph{constant mean curvature} (CMC) case \cite{changCAG}. In higher dimensions, the conjecture is still open, although some partial results have been obtained under additional conditions (see, e.g., the survey \cite{scherfner2012review} and the recent work \cite{spruck}). 
		
		In this paper, we focus on hypersurfaces immersed into five-dimensional space forms: the strategy is to exploit the unique features of four-dimensional Riemannian Geometry in order to obtain rigidity and topological results for these hypersurfaces. One of these features is the splitting of the bundle of $2$-forms into two subbundles $\Lambda_+$ and $\Lambda_-$ induced by the Hodge operator, which implies the decomposition of the \emph{Weyl tensor} $\operatorname{W}$, the trace-free, conformally invariant part of the Riemann tensor, into a self-dual and an anti-self-dual part, denoted respectively by $\operatorname{W}^+$ and $\operatorname{W}^-$; this allows one to write the Riemann curvature operator in an elegant and useful block form, where $\operatorname{W}^{\pm}$ appear as linear, symmetric operators $\mathcal{W}^{\pm}$ on the subbundles $\Lambda_{\pm}$ (see e.g \cite{besse}). 
		We say that a Riemannian four-manifold $(M^4,g)$ is \emph{half conformally flat} if either $\operatorname{W}^+$ or $\operatorname{W}^-$ vanish
		identically: in particular, $(M^4,g)$ is \emph{self-dual} (respectively, \emph{anti-self-dual}) if $\operatorname{W}^-\equiv 0$ 
		(respectively, if $\operatorname{W}^+\equiv 0$).
		Furthermore, if the manifold is also closed, the Euler characteristic $\chi(M)$ can be
		expressed \emph{via} the celebrated \emph{Chern--Gauss--Bonnet formula}, which can be written as
		\[
		32\pi^2\chi(M)=\int_M\left(\lvert\operatorname{W}\rvert^2-2\lvert\mathring{\operatorname{Ric}}\rvert^2+\dfrac{R^2}{6}\right)dV_g,
		\]
		where $\mathring{\operatorname{Ric}}$ is the traceless Ricci tensor and $R$ is the scalar curvature;
		this equation leads to strong topological results, for instance in the Einstein case \cite{hitchin}. We point out that, of course, the previous integral is conformally invariant and, in fact, the last two terms of the integrand define the second elementary symmetric function of the eigenvalues of the \emph{Schouten tensor}, which appears in the definition of the so-called $Q$-\emph{curvature} \cite{branson2008origins, chang2008q, fefferman2001q}. 
		
		In the context of four-dimensional hypersurfaces, the Weyl tensor and the Chern--Gauss--Bonnet formula will be the main ingredients of this paper: indeed, since the ambient spaces we consider are space forms $N^5(c)$ with constant sectional curvature $c$, the key observation is that both the Weyl tensor and the Chern--Gauss--Bonnet formula can be expressed in purely extrinsic terms, meaning that they only depend on the second fundamental form $A$
		(see \eqref{weylconst}, \eqref{weylplusminushyper}, \eqref{cherngaussbonnethyper}). In particular, the Chern--Gauss--Bonnet formula has the form
		\[
		32\pi^2\chi(M)=
		\int_M\left[
		3S^2-
		6\lvert A^2\rvert^2-
		\dfrac{16}{3}H^2S+H^4+8H\operatorname{tr}(A^3)+
		4c(6c-S+H^2)\right]dV_g,
		\] 
		where $S:=\lvert A\rvert^2$, $\lvert A^2\rvert^2$ is the squared norm of $A\circ A$ and $\operatorname{tr}(A^3)$ is the trace of 
		$A\circ A\circ A$.
		These considerations allow us to directly analyse the Weyl tensor, both locally and globally, and to exploit the integral formula for the Euler characteristic in order to derive topological conclusions and classification results. In particular, we make use of the extrinsic expressions of the traceless Ricci tensor and the Weyl tensor, which imply two key sharp inequalities that can be summarized as follows
		\[
		\dfrac{1}{4}S^2\leq\lvert A^2\rvert^2\leq \dfrac{7}{12}S^2:
		\]
		we point out that the right-hand side is a particular case of a more general algebraic inequality holding for symmetric, trace-free matrices obtained by Case and the second author in \cite{case2023sharp}.

		The paper is organised as follows: in Section \ref{prelim}, we provide a full description of the curvature components of a four-dimensional hypersurface in a space form and, in particular, we derive the explicit expression for the splitting of the Weyl tensor in this setting. This allows us to prove some local statements concerning the hypersurface, which remain true under ambient conformal changes. Hence, we may assume that, for these local results, the ambient space is a locally conformally flat manifold instead of a space form: indeed, we perform all the computations in the case of
		an ambient space form, for the sake of simplicity, and then we will use the conformal invariance of the statements to extend them
		to the locally conformally flat case (for further details, we 
		refer the reader to the discussion before Theorem \ref{metathm}).
		The first result we prove is the following:
		
		\begin{theorem} \label{equalweylnormslcf}
			Let $(M^4,g)$ be an oriented, isometrically immersed hypersurface in a locally conformally flat manifold $(N^5,g_N)$. Then, at every $p\in M$,        
			\begin{equation*}
				\lvert\operatorname{W}^+\rvert^2=
				\lvert\operatorname{W}^-\rvert^2=
				\dfrac{1}{2}\lvert\operatorname{W}\rvert^2.
			\end{equation*}
		\end{theorem}
		In the closed case, we recall that, by Poincaré duality, 
		the cup product on $2$-forms defines the so-called \emph{intersection form} as 
		the symmetric, bilinear form
		\begin{align*}
			H^2(M,\mathbb{R})\times H^2(M,\mathbb{R})&\longrightarrow\mathbb{R}\\
			([\omega_1],[\omega_2])&\longmapsto \int_M\omega_1\wedge\omega_2.
		\end{align*}
		Hence, one can consider the \emph{signature} $\tau(M)$ of the associated quadratic form,
		whose absolute value $\lvert\tau(M)\rvert$ is a topological invariant (it is easy to show that 
		$\tau(M)$ changes sign under orientation reversing diffeomorphisms). The
		decomposition $\Lambda^2=\Lambda_+\oplus\Lambda_-$ given by the action of the
		Hodge operator also induces a splitting of $H^2(M,\mathbb{R})$ into
		$H_+^2(M,\mathbb{R})$ and $H_-^2(M,\mathbb{R})$, which are
		the spaces of self-dual and anti-self-dual harmonic $2$-forms,
		respectively: this implies that $b_2=b_2^++b_2^-$, where $b_2$ is the
		second Betti number and 
		$b_2^{\pm}$ is the dimension of $H_{\pm}^2(M,\mathbb{R})$ as 
		a real vector space, and a simple observation shows that
		$\tau(M)=b_2^+-b_2^-$.
		Furthermore, by the well-known Hirzebruch's signature formula, on every closed 
		Riemannian four-manifold $(M,g)$ we have
		\begin{equation} \label{hirze}
			48\pi^2\tau(M)=\int_M\lvert\operatorname{W}^+\rvert^2-\lvert\operatorname{W}^-\rvert^2dV_g.
		\end{equation}
		By these considerations, Theorem \ref{equalweylnormslcf} and Poincaré duality immediately imply the following:
		\begin{cor} \label{signaturehyperlcf}
			Let $(M^4,g)$ be a closed, oriented, isometrically immersed hypersurface in a locally conformally
			flat manifold $(N^5,g_N)$. Then, 
			$\tau(M)=0$. Moreover, $\chi(M)$ is an even integer.
		\end{cor}
		\begin{exe}\label{CP2}
			We know that the complex projective plane 
			$\mathbb{CP}^2$ admits a K\"{a}hler-Einstein metric (the well-known
			\emph{Fubini--Study metric} $g_{FS}$). With the standard orientation, 
			$(\mathbb{CP}^2,g_{FS})$ is a self-dual manifold: moreover, Kuiper's Theorem \cite{kuiper} implies that $\mathbb{CP}^2$ does not admit any locally conformally flat metric, which
			means that, by \eqref{hirze}, $\tau(\mathbb{CP}^2)>0$. Furthermore, $\chi(\mathbb{CP}^2)=3$:
			by Corollary \ref{signaturehyperlcf}, $\mathbb{CP}^2$ 
			cannot be isometrically immersed in any $5$-dimensional
			locally conformally flat manifold. Another well-known example of manifolds which cannot be isometrically immersed in
			any $5$-dimensional locally conformally flat space is provided by K3 surfaces: indeed, they are all diffeomorphic to one another
			and they satisfy $\chi(K3)=24$ and $\lvert\tau(K3)\rvert$=16 (we recall that any such manifold admits a Ricci-flat, anti-self-dual metric, with respect to the orientation induced by the hyperk\"{a}hler structure, whose existence is guaranteed by Yau's resolution of the Calabi conjecture).
		\end{exe}
		
		We point out that the vanishing of $\tau(M)$ can also be proven in a more general setting. As explained in Remark \ref{pontrjagin}, 
		$\tau(M)=\frac{1}{3}p_1(M)$, where $p_1(M)$ is the first Pontrjagin class of $TM$: under the hypothesis of orientability of the ambient space
		and the hypersurface, if the first Pontrjagin class of the tangent bundle of the ambient space vanishes then $\tau(M)$ must be zero,
		showing that Corollary \ref{signaturehyperlcf} holds in a more general setting. However, we highlight the fact that Theorem \ref{equalweylnormslcf} shows the pointwise vanishing of the smooth function $\lvert\operatorname{W}^+\rvert^2-\lvert\operatorname{W}^-\rvert^2$ for hypersurfaces immersed in locally
		conformally flat spaces: up to now, it is not known if a closed, smooth manifold with vanishing
		signature admits a Riemannian metric such that $\lvert\operatorname{W}^+\rvert^2-\lvert\operatorname{W}^-\rvert^2\equiv 0$ (see Remark
		\ref{remavez}).
		
		The analysis of the Weyl operators $\mathcal{W}^{\pm}$ allows us to relate the number $m$ of distinct eigenvalues of $A$ and the number $w$ of distinct eigenvalues of $\mathcal{W}^{\pm}$: furthermore, we find a characterization of the isoparametric condition in terms of the eigenvalues of $\mathcal{W}^{\pm}$. The study of the eigenvalues of the (anti-)self-dual Weyl operator 
		has been addressed for many years and it lead to many interesting
		results in the context of four-dimensional Riemannian geometry 
		(see e.g. \cite{catinomastrimrn, Derdzinski, lebrunconfkahl, PengWu}). In our setting, we provide the following result:
		\pagebreak
		\begin{theorem} \label{eigentheoremlcf}
			Let $(M^4,g)$ be an oriented
			hypersurface isometrically immersed in a locally conformally flat manifold $(N^5,g_N)$:
			\begin{enumerate}
				\item 
				if $m$ is the number
				of distinct principal curvatures 
				and $w$ is the number of distinct
				eigenvalues of $\mathcal{W}^+$, then,
				at $p\in M$
				\begin{itemize}
					\item $w=3$
					if and only if $m=4$;
					\item $w=2$ if and only if 
					$m=3$ or there are two distinct principal curvatures, both with multiplicity two (in this case, $m=2$);
					\item $w=1$ if and only if 
					at least three principal curvatures are equal (and, hence, $m\leq 2$). In this case,
					$\operatorname{W}^+\equiv 0$
					at $p$;
				\end{itemize}
				\item let $(N^5,g_N)$ be a space form with constant sectional curvature $c$. If the scalar curvature $R$ and
				the mean curvature of $(M^4,g)$ are
				constant, then $(M^4,g)$ is
				isoparametric if and only if
				the eigenvalues of 
				$\mathcal{W}^+$ are
				constant functions. 
			\end{enumerate}
			The same results hold for $\mathcal{W}^-$.
		\end{theorem}
		
		\begin{rem} \label{remarklcf}
			\begin{enumerate}
				\item We point out that, if we do not
				assume that $R$ is constant,
				the second part of Theorem 
				\ref{eigentheoremlcf} might not
				be true: indeed, Otsuki showed
				that there exist infinitely many
				locally conformally flat minimal
				hypersurfaces with non-constant
				scalar curvature in $\mathbb{S}^{n+1}$
				which, consequently, are not isoparametric
				\cite{otsuki}, although they are not embedded (see \cite{wang2018simons} and the references therein). In fact, 
				the only locally conformally flat
				isoparametric hypersurfaces in 
				$\mathbb{S}^{n+1}$ are the
				totally geodesic equatorial spheres and the Clifford hypersurface
				$\mathbb{S}^1\left(\sqrt{\frac{1}{n}}\right)\times \mathbb{S}^{n-1}
				\left(\sqrt{\frac{n-1}{n}}\right)$. 
				\item In the minimal, isoparametric case, if $c=1$
				we can conclude that $w=2$ if
				and only if, at every point, there are two distinct principal curvatures, both with multiplicity two. This characterizes
				$(M^4,g)$ as a product 
				$M_1^2\times M_2^2$ of 
				two surfaces with constant curvature, since
				if $m=3$, $(M^4,g)$ would be one of the so-called \emph{Cartan minimal
					hypersurfaces}, which do not 
				exist in $\mathbb{S}^5$ (\cite{cartan}).
				\item We observe that, if the ambient space is $N^5(1)$, by the Gauss--Codazzi equations (see \eqref{scalconst} later), minimal, isoparametric hypersurfaces with $m=4$ provide nice examples of \emph{scalar-flat} manifolds which are not Ricci-flat: an explicit example was given by Cartan \cite{cartan40}, who showed that there exists an isoparametric family of immersions of $\mathbb{S}^1\times V_2(\mathbb{R}^3)$ in $\mathbb{S}^5$, one of which is minimal, with $m=4$ (this was later generalized by Nomizu to the case where the ambient manifold is $\mathbb{S}^{2n+1}$,
				for every $n\geq 2$ \cite{nomizu}). Here, $V_2(\mathbb{R}^3)=SO(3)\cong\mathbb{RP}^3$ is the Stiefel manifold of pairs of orthonormal vectors in $\mathbb{R}^3$ :
				in dimension four, this is the only isoparametric family
				of immersions with $m=4$. 
				
				We also recall that isoparametric hypersurfaces in spheres have been completely classified: for some references, see e.g. \cite{cecil, cecilryan, chi, FKM, ozeki1, ozeki2}. 
			\end{enumerate}
		\end{rem}
		By the local computations on the (anti-)self-dual Weyl tensor, we can provide an alternate proof, in dimension four, of a result due to Nishikawa and Maeda \cite{NiskikawaMaeda1974}, who proved that locally conformally flat hypersurfaces have a principal curvature of multiplicity at least $n-1$ at every point
		(see Proposition \ref{nishi}): we recall that, in the minimal case,  they are contained in catenoids, if the ambient space is complete and simply connected \cite{docarmodaj}.
		
		In Section \ref{topbounds}, we analyse the Euler characteristic of minimal hypersurfaces in space forms, in order to prove integral topological bounds. First, we recall that, given a closed Riemannian manifold of dimension $n$, the \emph{Yamabe constant} is defined as 
		the infimum over a conformal class $[g]$ of the so-called \emph{normalized Einstein--Hilbert functional}, i.e.
		\begin{equation} \label{yamabeinv}
			Y(M,[g]):=\inf_{g'\in [g]}\mathrm{Vol}_{g'}(M)^{-\frac{n-2}{n}}\int_MR_{g'}dV_{g'}>0.
		\end{equation}
		This quantity is a conformal invariant and it has been deeply studied over the years, in particular because of its connection with the problem of finding a metric of constant scalar curvature in an arbitrarily given conformal class of metrics on closed manifolds, known as the \emph{Yamabe problem}: in the closed case, such metrics always exist and they attain the minimum in \eqref{yamabeinv}, as shown by the efforts of Yamabe, Aubin, Trudinger and Schoen \cite{leepark}. 
		The sign of the Yamabe constant determines the sign of the scalar curvature of the metric in $[g]$ which attains the minimum: it follows that the existence of conformal classes with positive Yamabe constant might constrain the geometry and the topology of a smooth manifold (see e.g. \cite{bour2015optimal, bcdm, chang2003conformally, GurskyAnnals, Gurskyharmonicweyl, gurskylebrun, tran}).
		In the context of isometrically immersed hypersurfaces in $5$-dimensional space forms, we are able to prove an integral pinching result for the
		Euler characteristic, under the assumption of positive Yamabe constant:
		\begin{theorem} \label{boundyamabe}
			Let $(M^4,g)$ be a closed, oriented, minimal hypersurface into a space form $(N^5(c),g_N)$. Assume that $(M^4,g)$ has positive Yamabe constant. Then $c=1$ and either $M$ is diffeomorphic to $\mathbb{S}^4$ or
			\begin{equation} \label{eulerbound}
				\dfrac{1}{16}\int_M (4-S)(12+S)dV_g\leq 4\pi^2\chi(M) < \int_M \frac{1}{3}S^2dV_g;
			\end{equation}
			moreover, equality on the left-hand side is achieved if and only if $(M^4,g)$ is locally conformally flat
			and $\chi(M)\leq 0$.
		\end{theorem}
		After this, inspired by the optimal topological bounds for the $L^2$-norm of $\operatorname{W}$ (the so-called \emph{Weyl functional}) obtained by Gursky in \cite{GurskyAnnals} and \cite{gurskyzeit}, we derive sharp topological lower bounds for $||\operatorname{W}||_{L^2}^2$ in terms of $\chi(M)$ for minimal hypersurfaces immersed in space forms of zero or positive sectional curvature, in the constant scalar curvature case. More precisely, we use
		the extrinsic Chern--Gauss--Bonnet formula for minimal hypersurfaces \eqref{cherngaussbonnetminimal} and a local version of the rigidity
		results obtained by Chern, Do Carmo, Kobayashi \cite{cherndocarmokob} and Lawson \cite{lawson} (Proposition \ref{localisometryprop}) to prove the
		following:
		\pagebreak
		\begin{theorem} \label{topbound}
			Let $(M^4,g)$ be a closed, minimally immersed hypersurface
			in a space form $(N^5(c),g_N)$. Then, we have that
			\begin{enumerate}
				\item if $c=0$, either $(M^4,g)$ is locally conformally flat with $\chi(M)\geq 0$ or
				\begin{equation} \label{boundeucl}
					\int_M\lvert\operatorname{W}\rvert^2dV_g>\dfrac{256}{9}
					\pi^2\chi(M);
				\end{equation}
				\item if $c=1$ and $S$
				is constant, either $(M^4,g)$ is totally geodesic or 
				\begin{equation} \label{boundsphere}
					\int_M\lvert\operatorname{W}\rvert^2dV_g\geq\dfrac{64}{3}
					\pi^2\chi(M),
				\end{equation}
				where equality holds if and only if $M$ is locally isometric to \\ $\mathbb{S}^1(1/2)\times \mathbb{S}^3(\sqrt{3}/2)$ or to $\mathbb{S}^2(1/\sqrt{2})\times \mathbb{S}^2(1/\sqrt{2})$. 
			\end{enumerate}
		\end{theorem}
		This immediately implies the following 
		\begin{cor} \label{rigiditytopbound}
			Let $(M^4,g)$ be a closed, minimally immersed hypersurface with constant scalar curvature in a space form $(N^5(1),g_N)$. Then, if
			\[\int_M\lvert\operatorname{W}\rvert^2dV_g\leq \dfrac{64}{3}\pi^2\chi(M),
			\]
			either $(M^4,g)$ is totally geodesic or equality holds and $(M^4,g)$ is locally isometric to one of the following product manifolds:
			\begin{enumerate}
				\item $\mathbb{S}^1(1/2)\times \mathbb{S}^3(\sqrt{3}/2)$;
				\item $\mathbb{S}^2(1/\sqrt{2})\times \mathbb{S}^2(1/\sqrt{2})$.
			\end{enumerate}
			In particular, if the ambient space is $(\mathbb{S}^5, g_{\mathbb{S}^5})$, under the same hypotheses either $M$ is an equatorial sphere or isometric to a Clifford hypersurface. 
		\end{cor}
		Moreover, as a direct consequence of the method used in the proof of Theorem \ref{topbound}, we also obtain a lower topological bound for
		the Weyl functional in the constant, non-positive scalar curvature regime, analogous to the one obtained by Gursky \cite{gurskyzeit}: 
		\begin{corollary} \label{corbound}
			Let $(M^4,g)$ be a closed, minimally immersed hypersurface
			in a space form $(N^5(1),g_N)$. Then, if the scalar curvature
			$R$ of $(M^4,g)$ is constant and non-positive (i.e. $S\geq 12$), then
			\[
			\int_M\lvert\operatorname{W}\rvert^2dV_g\geq 
			32\pi^2\chi(M).
			\]
		\end{corollary}
		We point out that a lower bound for the Weyl functional in terms of the Betti numbers was found by Onti and Vlachos for conformal immersions in the standard Euclidean space \cite{onti}. 
		
		Finally, we address a problem derived from the weak version of the Chern conjecture, i.e. the so-called \emph{second pinching problem}: this more general version of the conjecture, also inspired by Simons' identity (\cite{simons}, see \eqref{simonsid}) and the
		aforementioned result by M\"{u}nzner \cite{munzner}, states that, if $M$ is a closed, $n$-dimensional minimal hypersurface with constant scalar curvature in the sphere with $S>n$, then $S\geq 2n$. The first step towards the resolution of this problem was taken in \cite{pengterng} and, since then, there have been many improvements: to the best of our knowledge, the best result so far was obtained by Yang and Cheng \cite{yangcheng}, who proved that, if $S>n$, then $S\geq n+C(n)$, where $C(n)=n/3$, but a full resolution of the problem has not been achieved yet. 
		Using the sharp integral bounds for the Weyl functional obtained in Theorem \ref{topbound}, we are able to provide a lower bound for 
		$S$ only depending on the Euler characteristic and the volume of $M$:
		\begin{cor} \label{corpinch}
			Under the hypothesis of case (2) of Theorem \ref{topbound}, if 
			$\chi(M)\geq 0$, then either $(M^4,g)$ is totally geodesic or we have 
			\begin{equation} \label{secondformbound}
				S\geq 4\pi\sqrt{\dfrac{\chi(M)}{\mathrm{Vol}_g(M)}}.
			\end{equation}
			Furthermore, equality holds if and only if $(M^4,g)$ is 
			locally isometric to $\mathbb{S}^2(1/\sqrt{2})\times \mathbb{S}^2(1/\sqrt{2})$.
		\end{cor}
		In fact, by analyzing the Chern--Gauss--Bonnet formula and using sharp algebraic estimates for the 
		powers of the second fundamental form, we can obtain the following:
		\begin{theorem} \label{sboundeuler}
			Let $(M^4,g)$ be a closed, minimally immersed, non-totally geodesic hypersurface in a 
			space form $(N^5(1),g_N)$ with constant scalar curvature. Then, 
			\[
			S\geq f\left(\dfrac{\chi(M)}{\mathrm{Vol_g(M)}}\right),
			\]
			where 
			\[
			f\left(\dfrac{\chi(M)}{\mathrm{Vol_g(M)}}\right)=
			\begin{cases}
				-4+8\sqrt{1-\dfrac{\pi^2\chi(M)}{\mathrm{Vol}_g(M)}}, &\mbox{ if }
				\dfrac{\chi(M)}{\mathrm{Vol}_g(M)}\leq\dfrac{9}{25\pi^2},\\
				4\pi\sqrt{\dfrac{\chi(M)}{\mathrm{Vol}_g(M)}}, &\mbox{ if }
				\dfrac{9}{25\pi^2}<\dfrac{\chi(M)}{\mathrm{Vol}_g(M)}\leq
				\dfrac{1}{\pi^2},\\
				\dfrac{4}{3}+\dfrac{8\sqrt{2}}{3}\sqrt{\dfrac{3\pi^2\chi(M)}{2\mathrm{Vol}_g(M)}-1}, &\mbox{ if }
				\dfrac{1}{\pi^2} <\dfrac{\chi(M)}{\mathrm{Vol}_g(M)}.
			\end{cases}
			\]
			Moreover, 
			\begin{enumerate}
				\item $S=-4+8\sqrt{1-\frac{\pi^2\chi(M)}{\mathrm{Vol}_g(M)}}$ if and only if $(M^4,g)$ is locally isometric to 
				$\mathbb{S}^1(1/2)\times\mathbb{S}^3(\sqrt{3}/2)$;
				\item $S=\frac{4}{3}+\frac{8\sqrt{2}}{3}\sqrt{\frac{3\pi^2\chi(M)}{2\mathrm{Vol}_g(M)}-1}$ if and only if $S=4\pi\sqrt{\frac{\chi(M)}{\mathrm{Vol}_g(M)}}$ if and only if $(M^4,g)$ is locally isometric to $\mathbb{S}^2(1/\sqrt{2})\times \mathbb{S}^2(1/\sqrt{2})$.
			\end{enumerate}
		\end{theorem}
		
		We also exploit volume estimates proven in \cite{chenli} to obtain lower bounds for $S$ in terms of the Euler characteristic $\chi(M)$,
		in the case where the ambient space is $\mathbb{S}^5(1)$ immersed in the standard way in $\mathbb{R}^6$, centered at the origin, 
		and under an assumption on the cross-sectional areas of $M$. We highlight the fact that when $\chi(M)\notin\{0,2,4 \}$ these lower bounds for $S$ improve the best estimate in the literature,
		which, for $n=4$, is $S\geq 4+4/3=16/3$.
		More precisely, we can prove the following 
		\begin{cor} \label{corsboundeuler}
			Under the hypotheses of Theorem \ref{sboundeuler}, if 
			$(N^5(1),g_N)$ is the standard sphere $(\mathbb{S}^5,g_{\mathbb{S}^5})$,
			let $Z_a^M:=M\cap\mathbb{S}_a^4$, where $\mathbb{S}_a^4$ is the intersection
			of $\mathbb{S}^5$ with the hyperplane of $\mathbb{R}^6$ consisting
			of vectors orthogonal to $a\in\mathbb{S}^5$. If
			\[
			\sup_{a\in\mathbb{S}^5} \mathrm{Vol}_g(Z_a^M)\leq\lvert\mathbb{S}^4\rvert,
			\]
			then 
			\[
			S \geq \begin{cases}
				-4+8\sqrt{1-\dfrac{4\chi(M)}{5\pi}}, \quad &\mbox{ if }
				\chi(M)\leq0 \\
				\dfrac{4}{3}+\dfrac{8\sqrt{2}}{3}\sqrt{\dfrac{6\chi(M)}{5\pi}-1}, \quad &\mbox{ if } \chi(M)>2.
			\end{cases}
			\]
			In particular, $S>4+C(4)=16/3$ if $\chi(M)\not\in\{0,2,4\}$.
		\end{cor}
		Lastly, in Section \ref{bochnersection}, we derive Bochner--Weitzenb\"{o}ck formulas for the second fundamental form and exploit them to derive rigidity results in terms of integral inequalities. First, we provide sharp integral bounds for a general minimal hypersurface in a positive space form $(N^{n+1},g_N)$, also characterizing the totally geodesic and Clifford-type hypersurfaces:
		\begin{theorem} \label{sharpquadrbound}
			Let $(M^n,g)$ be a closed, minimally immersed hypersurface in a space form $(N^{n+1}(1),g_N)$, with $n\geq 4$. Then, 
			\small
			\begin{equation} \label{boundasquare}
				\int_M\lvert\nabla S\rvert^2+2\left(\lvert A^2\rvert^2-\sqrt{\dfrac{n^2-3n+3}{n(n-1)}}S^2\right)(S-n)dV_g\leq\int_M\lvert\nabla A^2\rvert^2dV_g\leq \dfrac{4}{3}\int_M\lvert A^2\rvert^2(S-n)dV_g.
			\end{equation}
			\normalsize
			Moreover, equality holds on either side of \eqref{boundasquare} if and only if $(M^n,g)$ is either totally geodesic or locally isometric to a Clifford hypersurface. 
		\end{theorem}
		As an easy consequence, we also provide estimates in the case where $S\geq n$:
		\begin{cor} \label{sharpquadrcor}
			If $S\geq n$, then
			\small
			\begin{equation} \label{refinedboundasquare}
				\int_M\lvert\nabla S\rvert^2-2\left(\dfrac{1}{n}-\sqrt{\dfrac{n^2-3n+3}{n(n-1)}}\right)S^2(S-n)dV_g\leq\int_M\lvert\nabla A^2\rvert^2dV_g\leq 
				\dfrac{4(n^2-3n+3)}{3n(n-1)}\int_MS^2(S-n)dV_g.
			\end{equation}
			\normalsize
			Moreover, equality holds on either side of \eqref{refinedboundasquare} if and only if $(M^n,g)$ is locally isometric to a Clifford hypersurface. 
		\end{cor}
		Then, we study natural conditions on the Weyl tensor in the four-dimensional case which are weaker than local conformal flatness: first, we deal with \emph{half harmonic Weyl} metrics, i.e. metrics whose (anti)-self-dual Weyl tensor is divergence-free. This condition is particularly interesting because it is satisfied both by Einstein and half conformally flat metrics and it is also intimately related to the geometry of K\"{a}hler surfaces
		(see e.g. \cite{garciario, Derdzinski, lebrunconfkahl, PengWu, wuwuwylie}). So far, to the best of our knowledge, no classification result for half harmonic Weyl hypersurfaces has been proven (although, the study of harmonic curvature hypersurfaces has lead to strong conclusions, see e.g. \cite{umehara}). By means of a famous Bochner formula derived by Derdzi\'{n}ski \cite{Derdzinski}, we obtain the following integral rigidity result 
		\begin{theorem} \label{harmweylinequality}
			Let $(M^4,g)$ be a closed, oriented, minimally immersed hypersurface in a space form $(N^5(c),g_N).$ Then, if $(M^4,g)$ is half harmonic Weyl, then either 
			$(M^4,g)$ is locally conformally flat or the 
			following inequality holds:
			\begin{equation} \label{sharpinequality}
				c\int_M\left(\dfrac{7}{3}S^2-
				4\lvert A^2\rvert^2\right)dV_g\leq\int_M\left(
				5[\operatorname{tr}(A^3)]^2-14\operatorname{tr}(A^6)
				+\dfrac{55}{6}S
				\lvert A^2\rvert^2-\dfrac{13}{12}S^3\right)dV_g.
			\end{equation}
			Moreover, 
			the equality in \eqref{sharpinequality} is achieved if 
			and only if $c=1$ and $(M,g)$ is locally isometric to the Clifford hypersurface $
			\mathbb{S}^2(1/\sqrt{2})\times \mathbb{S}^2(1/{\sqrt{2}})$.
		\end{theorem}
		
		Finally, we deal with \emph{Bach-flat} metrics on four-dimensional hypersurfaces, i.e. critical points of the Weyl functional, which is $||\operatorname{W}||_{L^2}^2$. The Bach tensor $\operatorname{B}$ is a symmetric, trace-free $(0,2)$-tensor, which, in dimension four, also happens to be conformally covariant: up to now, the
		relations between the existence of Bach-flat metrics and the topology of the underlying four-manifold are not well understood, although some strong rigidity and classification results have been obtained (see e.g. \cite{caocatino, caochen, Chang-gursky-yang-annals, chang2003conformally, lebrunkahler}). Examples
		of Bach-flat metrics are provided by conformally Einstein and half conformally flat metrics, although, to the best of our knowledge, 
		there only exists one example which does not belong to either of these categories \cite{abbena}. 
		In this paper, we start the study of the Bach-flat condition for 
		a minimal hypersurface immersed in a space form: we recover the 
		well-known Weitzenb\"{o}ck formula derived in 
		\cite{Chang-gursky-yang-annals} (see \eqref{secondbochnerbach}) 
		and, by means of a well-known algebraic inequality for 
		trace-free, symmetric matrices \cite{case2023sharp, huisken},
		we prove two sharp integral inequalities for the second 
		fundamental form: 
		\begin{theorem} \label{bachident}
			Let $(M^4,g)$ be a closed, oriented, minimally immersed
			hypersurface in a space form $(N^5(c),g_N)$. Assume that $(M,g)$
			is a Bach-flat manifold. Then, 
			the following identities hold:
			\begin{align}
				A_{ij}A_{ikl}A_{jkl}&=\operatorname{tr}(A^5)-
				\left(2c+\dfrac{1}{3}S\right)\operatorname{tr}(A^3)+\dfrac{1}{6}A_{ij}S_{ij};
				\label{firstbochnerbach} \\
				\dfrac{1}{2}\Delta\lvert A^2\rvert^2&=
				\lvert\nabla A^2\rvert^2+2\operatorname{tr}(A^6)-
				2\left[\operatorname{tr}(A^3)\right]^2-
				\dfrac{7}{6}S\lvert A^2\rvert^2+
				\label{secondbochnerbach}\\ 
				&+\dfrac{1}{6}S^3+c\left(4
				\lvert A^2\rvert^2-S^2\right)+
				\dfrac{1}{3}S_{ij}A_{ij}^2+\dfrac{1}{6}S\Delta S\notag,
			\end{align}
			where $S_{ij}$ are the local components of $\operatorname{Hess}S$.
			In particular, we obtain the following integral inequalities:
			\begin{equation} \label{bachintegral1}
				-\dfrac{5\sqrt{3}}{18}\int_MS^{\frac{5}{2}}dV_g\leq\int_M\operatorname{tr}(A^5)dV_g\leq 
				\dfrac{5\sqrt{3}}{18}\int_MS^{\frac{5}{2}}dV_g;
			\end{equation}
			\begin{equation} \label{bachintegral2}
				\int_M\lvert\nabla A^2\rvert^2-\dfrac{1}{3}\lvert\nabla S\rvert^2dV_g\leq 
				\int_M\left[\dfrac{85}{72}S^3-2\operatorname{tr}(A^6)-
				c(4\lvert A^2\rvert^2-S^2)\right]dV_g.
			\end{equation}
			Moreover, equality holds in either side of \eqref{bachintegral1} and in \eqref{bachintegral2} if and only if $(M^4,g)$ is locally conformally flat. 
		\end{theorem}
		It is worth noting that, since all locally conformally flat metrics are
		Bach-flat, the aforementioned result by Nishikawa and Maeda (Proposition \ref{nishi}) and Theorem \ref{bachident} immediately
		provide an integral equality for the second fundamental form
		of locally conformally flat, minimal hypersurfaces in space forms:
		\begin{cor} \label{corbachlcf}
			Let $(M^4,g)$ be a closed, minimally immersed hypersurface in a space form 
			$(N^5(c),g_N)$. If $g$ is a locally conformally flat
			metric, then
			\[
			\int_M\lvert\nabla A^2\rvert^2-\dfrac{1}{3}\lvert\nabla S\rvert^2dV_g=
			\dfrac{1}{3}\int_MS^2(S-4c)dV_g.
			\]
		\end{cor}
		Finally, we can also provide a strict integral inequality in the half harmonic Weyl case and a characterization of
		totally geodesic space forms inside $N^5(1)$, under the Bach-flat hypothesis:
		\pagebreak
		\begin{cor} \label{rigresultintegral}
			\begin{enumerate}
				\item Under the hypotheses of Theorem \ref{harmweylinequality}, either 
				$(M^4,g)$ is locally conformally flat or
				\begin{equation} \label{strictharmweyl}
					c\int_M\left(\dfrac{7}{3}S^2-4\lvert A^2\rvert^2\right)dV_g < \int_M\left(\dfrac{427}{72}S^3-14\operatorname{tr}(A^6)\right)dV_g.
				\end{equation}
				\item Under the hypotheses of Theorem \ref{bachident}, if $c=1$ then
				\begin{equation} \label{totgeodbach}
					\int_M\left(\lvert\nabla A^2\rvert^2-\dfrac{1}{3}\lvert\nabla S\rvert^2\right)dV_g\leq
					\int_M\left(
					\dfrac{85}{72}S^3-2\operatorname{tr}(A^6)\right)dV_g;
				\end{equation}
				moreover, equality holds if and only if 
				$(M^4,g)$ is totally geodesic.
			\end{enumerate}
		\end{cor}
		
		\subsection{Acknowledgements}
		This work was initiated during the International Doctoral Summer School In Conformal Geometry and Non-local Operators at the Instituto de Matem\'aticas de la Universidad de Granada (IMAG). We thank the organizers, Azahara DelaTorre Pedraza and Mar\'ia del Mar Gonz\'alez, for their kind invitation. We would also like to thank Jeffrey S. Case, Giovanni Catino, Matthew J. Gursky, Paolo Mastrolia and Andrea Seppi for many helpful conversations and suggestions. The first author is a member of the Gruppo Nazionale per le Strutture Algebriche, Geometriche e loro Applicazioni (GNSAGA) of INdAM (Istituto Nazionale di Alta Matematica).
		\section{Four-Dimensional hypersurfaces in space forms and the Weyl tensor} \label{prelim}
		
		Let $N^{n+1}$ be a smooth, connected, orientable manifold of dimension $n+1$ and suppose there exists a Riemannian metric $g_N$ such that 
		$(N^{n+1}(c),g_N)$ is a space form with constant sectional curvature $c$: from now on, we will assume that $g_N$ is normalized in such a way that $c\in\{-1,0,1\}$. 
		Let $M^n$ be a smooth, connected,  closed manifold of dimension $n$ and 
		$\iota:M^n\longrightarrow (N^{n+1}(c),g_N)$ be an immersion. 
		We can define a Riemannian metric $g$ induced by $g_N$
		as $g=\iota^{\ast}(g_N)$: in this case, we say that the map
		\[
		\iota:(M^n,g)\longrightarrow (N^{n+1}(c),g_N)
		\]
		is an \emph{isometric immersion} of $M^n$ into $N^{n+1}$. Throughout the paper, we sometimes omit the dimensional notation on $M$ and $N$, when there is no ambiguity:
		moreover, we will use Einstein's summation convention over repeated indices, unless specified otherwise, and we will assume that $M^n$ is oriented with the orientation induced by
		the one on $N^{n+1}$ (however, we point out that some of the local results hold even in the non-orientable case).
		
		We want to compute the curvature of the induced metric $g$ on $M$:
		we can choose a local Darboux frame along $\iota$, i.e. a local orthonormal frame $\{E_a\}_{a=1}^{n+1}$ on $N$ 
		such that $\{E_i\}_{i=1}^n$ locally span the image of the tangent
		bundle $TM$ \emph{via} the pushforward $\iota_{\ast}$ and $E_5$ locally
		spans the normal bundle $TM^{\perp}$. This allows
		us to construct a local orthonormal frame $\{e_i\}_{i=1}^n$ on $M$ 
		such that $\iota_{\ast}(e_i)=E_i$
		(for a detailed description
		see, for instance, \cite{aliasmastroliarigoli}).
		
		We consider the coframe $\{\theta^a\}_{a=1}^{n+1}$ dual to $\{E_a\}$, which,
		by our choices, is such that
		\[
		\iota^{\ast}(\theta^{n+1})=0.
		\]
		Exploiting Cartan's structure equations for $\iota^{\ast}(\theta^i)$ (see e.g. \cite{CatinoMastroliaBook}),
		we get the local expression of the Riemann tensor $\operatorname{Riem}$ of $g$ (Gauss--Codazzi equations):
		\begin{equation} \label{riem}
			R_{ijkl}={}^NR_{ijkl}+A_{ik}A_{jl}-A_{il}A_{jk},
		\end{equation}
		where ${}^NR_{ijkl}$ denotes the components of the Riemann tensor ${}^N\operatorname{Riem}$ of
		$(N,g_N)$ and $A_{ij}$ are the components of the second fundamental
		form $A$ with respect to the chosen Darboux frame. In accordance with much of the literature, throughout the paper we will denote the
		squared norm of $A$ as
		\[
		\lvert A\rvert^2=A_{ij}A_{ij}=:S.
		\]
		Since $g_N$ is a constant sectional curvature metric, we can
		rewrite \eqref{riem} as
		\begin{equation} \label{riemconst}
			R_{ijkl}=c(\delta_{ik}\delta_{jl}-\delta_{il}\delta_{jk})
			+A_{ik}A_{jl}-A_{il}A_{jk}
		\end{equation}
		(recall that, since the components are computed with respect to
		a local orthonormal coframe, $g_{ij}=\delta_{ij}$).
		By \eqref{riemconst}, we also obtain the expressions 
		for the Ricci tensor $\operatorname{Ric}$ and the scalar
		curvature $R$
		\begin{align}
			R_{ij}&=(n-1)c\delta_{ij}+HA_{ij}-A_{it}A_{tj} \label{ricconst}\\
			R&=n(n-1)c+H^2-S \label{scalconst},
		\end{align}
		where $H=\sum_{i=1}^{n}A_{ii}$ is the mean curvature
		of the immersion: we point out that, in much of the literature, the mean curvature is defined as $(1/n)H$. 
		From now on, we will write
		$A^m=\underbrace{A\circ...\circ A}_\text {$m$ times}$ and 
		\[
		A_{ij}^m=A_{ik_1}A_{k_1k_2}\cdots
		A_{k_{m-1}k_m}A_{k_mj}.
		\]
		Since $A$ is a symmetric $(0,2)$-tensor,
		at every $p\in M$ there exists a 
		local orthonormal frame such that
		$A$ is in diagonal form, i.e. 
		\begin{equation} \label{diagonalsecond}
			A_{ij}=\lambda_i\delta_{ij}
		\end{equation}
		(here we are not summing over $i$): 
		we say that the eigenvalues $\lambda_i$
		of $A$ are the \emph{principal curvatures} of the isometric immersion and
		a hypersurface with constant principal 
		curvatures is called \emph{isoparametric}.
		We also recall that, if $M$ is \emph{minimal} (i.e. $H\equiv 0$ on $M$), by well-known commutation formulas for the second derivatives of $(0,2)$-tensors (\cite{CatinoMastroliaBook}), it is possible to show that 
		\begin{equation} \label{simonsidlocal}
			\Delta A_{ij}=(nc-S)A_{ij},
		\end{equation}
		which, by contraction with $A_{ij}$, implies the celebrated \emph{Simons' identity} \cite{simons}
		\begin{equation} \label{simonsid}
			\dfrac{1}{2}\Delta S=\lvert\nabla A\rvert^2+S(nc-S).
		\end{equation}
		If $M$ is closed, integrating \eqref{simonsid} with $c=1$ we obtain 
		\begin{equation} \label{simonsidint}
			\int_M\lvert\nabla A\rvert^2dV_g=\int_MS(S-n)dV_g,
		\end{equation}
		which implies that, if $S\leq n$, then either $(M,g)$ is totally geodesic (i.e. $A\equiv 0$ on $M$) or $S\equiv n$ on $M$: the latter case was completely characterized by Chern, do Carmo, Kobayashi and Lawson \cite{cherndocarmokob,lawson}. It is
		worth mentioning that Simons-type equalities were also found, for instance, by Nomizu and Smyth 
		\cite{nomizusmyth} and Wang \cite{wang2018simons},
		in order to provide characterizations of special hypersurfaces (e.g. Clifford minimal hypersurfaces or catenoids). 
		
		Now, we know that, on every Riemannian manifold of dimension $n\geq 3$,
		the following decomposition holds
		\begin{equation} \label{riemdeco}
			R_{ijkl}=W_{ijkl}+\dfrac{1}{n-2}
			(R_{ik}\delta_{jl}-R_{il}\delta_{jk}+
			R_{jl}\delta_{ik}-R_{jk}\delta_{il})-
			\dfrac{R}{(n-1)(n-2)}(\delta_{ik}\delta_{jl}-\delta_{il}\delta_{jk}),
		\end{equation}
		where $W_{ijkl}$ are the components of the \emph{Weyl tensor}
		$\operatorname{W}$. We recall that the Weyl tensor enjoys the same symmetries as $\operatorname{Riem}$ and it is totally trace-free: furthermore, $\operatorname{W}$ is the \emph{conformally invariant} part of the curvature, since, under a conformal transformation $\widetilde{g}=e^{2u}g$, where $u\in C^{\infty}(M)$, its $(1,3)$ version satisfies 
		\[
		\widetilde{\operatorname{W}}=\operatorname{W}.
		\]
		
		From now on, let $n=4$: by \eqref{riemconst}, \eqref{ricconst},
		\eqref{scalconst} and \eqref{riemdeco}, we can compute
		$W_{ijkl}$ as
		\begin{align} \label{weylconst}
			W_{ijkl}&=A_{ik}A_{jl}-A_{il}A_{jk}
			-\frac{H}{2}(A_{ik}\delta_{jl}-A_{il}\delta_{jk}
			+A_{jl}\delta_{ik}-A_{jk}\delta_{il})+\\
			&+\dfrac{1}{2}(A^2_{ik}\delta_{jl}-A^2_{il}\delta_{jk}
			+A^2_{jl}\delta_{ik}-A^2_{jk}\delta_{il})
			+\dfrac{1}{6}(H^2-S)
			(\delta_{ik}\delta_{jl}-\delta_{il}\delta_{jk}) \notag. 
		\end{align}
		We point out that the expression in \eqref{weylconst} does not
		depend on the sectional curvature $c$ of $N$, reflecting the conformal invariance of $\operatorname{W}$. 
		
		Now, since $M$ is oriented, if $\eta$ is a volume form on $N$ locally 
		defined as
		\[
		\eta=\theta^1\wedge...\wedge\theta^5,
		\]
		we can locally define a volume form $\mu\in\Lambda^4M$ as
		\[
		\mu=\iota^{\ast}(\theta^1)\wedge...\wedge\, \iota^{\ast}(\theta^4).
		\]
		For the sake of simplicity, we will omit the pullback notation from
		now on. Since we chose a local Darboux frame, which, in particular,
		is orthonormal, we can rewrite the local expression of $\mu$ as
		\begin{equation} \label{volform}
			\mu=\dfrac{1}{24}\mu_{ijkl}\theta^i\wedge\theta^j\wedge\theta^k
			\wedge\theta^l,
		\end{equation}
		where $\mu_{ijkl}$ is the \emph{Levi-Civita symbol}, i.e.
		\begin{equation} \label{levicivitasymbol}
			\mu_{ijkl}=
			\begin{cases}
				\operatorname{sgn}(\sigma), &\mbox{ if } 
				\sigma \mbox{ is a permutation of } (1234)\\
				0 &\mbox{ otherwise}
			\end{cases}.
		\end{equation}
		Recall that, in general, \eqref{levicivitasymbol} 
		is skew-symmetric in all the indices and it satisfies
		the following important property
		\begin{equation} \label{symbolcontraction}
			\mu_{i_1...i_k i_{k+1}...i_4}\mu_{i_1...i_kj_{k+1}...j_4}=
			k!\delta_{i_{k+1}...i_4}^{j_{k+1}...j_4},
		\end{equation}
		where $\delta_{i_1,...,i_p}^{j_1,...j_p}$ is the
		\emph{generalized Kronecker delta}
		defined as
		\[
		\delta_{i_1\cdots i_p}^{j_1\cdots j_p}=
		\begin{cases}
			\operatorname{sgn}(\sigma), 
			\quad &\mbox{ if } i_1,...,i_p \mbox{ are all different and } (j_1,...,j_p)=\sigma(i_1,...,i_p), \mbox{ for } \sigma\in S_p\\
			$0$ \quad &\mbox{ otherwise }
		\end{cases}
		\]
		(here $S_p$ denotes the $p$-th symmetric group).
		We define the 
		\emph{Hodge star operator} $\star$ as the map
		\begin{equation} \label{hodge}
			\star:\Lambda^k\longrightarrow\Lambda^{4-k},
		\end{equation}
		such that, for every $\omega\in\Lambda^k$, $\star(\omega)$ is the unique $(4-k)$-form such that
		$\eta\wedge\star(\omega)=\langle\eta,\omega\rangle\mu$, for every $\eta\in\Lambda^k$. 
		
		From now on, we consider $\star$ applied to $\omega\in\Lambda^2M$: 
		in this case, $\star$ is an automorphism of $\Lambda^2M$ and,
		since locally
		$\omega=\omega_{ij}\theta^i\wedge\theta^j$ for some smooth
		functions $\omega_{ij}$, we can locally express \eqref{hodge}
		using \eqref{volform} and \eqref{levicivitasymbol} as
		\begin{equation} \label{hodge2local}
			\star(\omega)=\dfrac{1}{4}\mu_{ijkl}\omega_{kl}\theta^i\wedge\theta^j.
		\end{equation}
		Since $\operatorname{W}$ satisfies the same symmetries as the
		Riemann tensor $\operatorname{Riem}$, 
		we can define the \emph{Weyl operator} as
		\begin{align*}
			\mathcal{W}: \Lambda^2&\longrightarrow\Lambda^2\\
			\omega &\longmapsto \mathcal{W}\omega,
		\end{align*}
		where $\mathcal{W}\omega$ is the 2-form whose
		local expression is
		\[
		\mathcal{W}\omega=\dfrac{1}{4}W_{ijkl}\omega_{kl}
		\theta^i\wedge\theta^j
		\]
		We can 
		consider the compositions $\star\mathcal{W}$ and 
		$\mathcal{W}\star$; however, it is easy to see that
		\begin{equation} \label{weylhodge}
			\star\mathcal{W}=\mathcal{W}\star.
		\end{equation}
		Hence, we can define a new operator
		\begin{align*}
			\star\mathcal{W}: \Lambda^2&\longrightarrow\Lambda^2\\
			\omega &\longmapsto \star\mathcal{W}(\omega)=
			\mathcal{W}(\star\omega):
		\end{align*}
		the $2$-form $\star\mathcal{W}(\omega)$ has
		the local expression
		\[
		\star\mathcal{W}(\omega)=\dfrac{1}{4}(\star W)_{ijkl}
		\omega_{kl}\theta^i\wedge \theta^j,
		\]
		where
		\begin{equation} \label{locstarweyl}
			(\star W)_{ijkl}=\dfrac{1}{2}\mu_{ijrs}W_{klrs}.
		\end{equation}
		Observe that we can identify the operator $\star\mathcal{W}$ with a $(0,4)$-tensor $\star\operatorname{W}$, whose local components are defined in \eqref{locstarweyl}.
		
		Since $\{\theta^i\}$ is a local orthonormal coframe, using \eqref{locstarweyl}
		we can write the following useful identities 
		for the local components $(\star W)_{ijkl}$:
		\begin{align} \label{starweyl}
			(\star W)_{1212}&=(\star W)_{3434}=W_{1234};\\
			(\star W)_{1213}&=(\star W)_{3442}=W_{1242}=W_{1334};\notag\\
			(\star W)_{1214}&=(\star W)_{2334}=W_{1223}=W_{1434};\notag\\
			(\star W)_{1313}&=(\star W)_{4242}=W_{1342};\notag\\
			(\star W)_{1314}&=(\star W)_{2342}=W_{1323}=W_{1442};\notag\\
			(\star W)_{1414}&=(\star W)_{2323}=W_{1423}. \notag
		\end{align}
		For the other components, recall that, in this case, 
		the Hodge operator is an involution, i.e. 
		$\star^2=\operatorname{Id}_{\Lambda^2}$: therefore, it is
		easy to see that, if $(\star W)_{ijkl}=W_{i'j'k'l'}$, then
		$(\star W)_{i'j'k'l'}=W_{ijkl}$.  
		
		It is well known that, given a Riemannian manifold 
		$(M,g)$ of dimension $4$, the Hodge star operator 
		induces a splitting of the bundle $\Lambda^2$ into two
		subbundles $\Lambda_+$ and $\Lambda_-$, called the 
		bundles of the \emph{self-dual} and the \emph{anti-self-dual} $2$-forms, 
		respectively.
		More precisely, $\Lambda_{\pm}$ is the eigenspace of
		$\star$ corresponding to the eigenvalue $\pm 1$: 
		therefore, a $2$-form $\omega$ is self-dual if
		$\star\omega=\omega$ and it is anti-self-dual if
		$\star\omega=-\omega$. Note that every $2$-form
		$\eta$ can be decomposed into a self-dual and an
		anti-self-dual part as
		\begin{equation} \label{formdeco}
			\eta=\dfrac{1}{2}(\eta+\star\eta)+\dfrac{1}{2}(\eta-
			\star\eta).
		\end{equation}
		As a consequence, the Weyl tensor can be decomposed as
		\begin{equation} \label{weyldecomp}
			\operatorname{W}=\operatorname{W}^++\operatorname{W}^-;
		\end{equation}
		we say that $\operatorname{W}^+$ (resp., $\operatorname{W}^-$) 
		is the \emph{self-dual}
		(resp., \emph{anti-self-dual}) part of the Weyl tensor.
		We say that $(M,g)$ is \emph{self-dual} (resp., \emph{anti-self-dual})
		if $\operatorname{W}^-$ (resp., $\operatorname{W}^+$) identically
		vanishes on $M$; in general, if one of the two addenda in 
		\eqref{weyldecomp} vanishes, we say that $(M,g)$ is 
		\emph{half conformally flat}.
		
		We can describe the decomposition \eqref{weyldecomp} more clearly. If we apply the operator $\star \mathcal{W}$ to
		a self-dual form $\omega$, by \eqref{weylhodge} we obtain that 
		\[
		\star\mathcal{W}(\omega)=\mathcal{W}(\star\omega)=
		\mathcal{W}(\omega),
		\]
		which means that $\star\mathcal{W}(\omega)$ is 
		a self-dual form as well. 
		Analogously, if $\omega$ is an anti-self-dual form,
		$\star\mathcal{W}(\omega)=-\mathcal{W}(\omega)$:
		in particular, this means that
		\[
		\mathcal{W}(\Lambda_{\pm})\subset\Lambda_{\pm}.
		\]
		This allows us to define the self-dual and the
		anti-self-dual Weyl operators as
		\begin{equation} \label{weylplusminus}
			\mathcal{W}^{\pm}=\dfrac{1}{2}(\mathcal{W}\pm\star \mathcal{W}),
		\end{equation} 
		which can be locally expressed as 
		$\mathcal{W}^{\pm}\omega=\frac{1}{4}W_{ijkl}^{\pm}\omega_{kl}\theta^i\wedge\theta^j$, where
		\begin{equation} \label{weylplusminuscomp}
			W_{ijkl}^{\pm}:=\dfrac{1}{2}[W_{ijkl}\pm(\star W)_{ijkl}]=
			\dfrac{1}{2}\left[W_{ijkl}\pm\dfrac{1}{2}\mu_{ijpq}W_{pqkl}\right];
		\end{equation}
		note that, by definition,
		\[
		\mathcal{W}^{\pm}(\Lambda_{\pm})\subset \Lambda_{\pm}, \quad \mathcal{W}^{\pm}(\Lambda_{\mp})=0.
		\]
		By \eqref{weylplusminus} and the fact that 
		$\star^2=\operatorname{Id}_{\Lambda^2}$, it follows that
		\begin{equation} \label{hodgeweylplusmin}
			\star(\mathcal{W}^{\pm})=\dfrac{1}{2}[
			\star(\mathcal{W}\pm\star\mathcal{W})]=
			\dfrac{1}{2}[\star\mathcal{W}\pm\star(\star\mathcal{W})]=
			\pm\mathcal{W}^{\pm},
		\end{equation}
		which can be written locally as
		\begin{equation} \label{hodgeweylplusmincomp}
			W_{ijkl}^{\pm}=\pm\dfrac{1}{2}\mu_{ijpq}W_{pqkl}^{\pm}.
		\end{equation}
		It is therefore natural to define two tensors 
		$\operatorname{W}^+$ and $\operatorname{W}^-$ as
		\[
		\operatorname{W}^{\pm}=
		\dfrac{1}{2}(\operatorname{W}\pm 
		(\star \operatorname{W}))
		\]
		which is the tensorial equivalent of 
		\eqref{weylplusminus}: indeed, the local components of $\operatorname{W}^{\pm}$ are exactly the ones defined in \eqref{weylplusminuscomp} and this also immediately shows that $\operatorname{W}$ decomposes as 
		in \eqref{weyldecomp}. 
		
		By \eqref{hodgeweylplusmincomp}, it is straightforward to prove the 
		orthogonality of the tensors $\operatorname{W}^{+}$ and 
		$\operatorname{W}^{-}$, i.e.
		\begin{equation} \label{weylplusminorthog}
			W_{ijkl}^{+}W_{ijkl}^{-}=0.
		\end{equation}
		An immediate consequence of \eqref{weylplusminorthog} is the
		validity of the following identities:
		\begin{align}
			\lvert\operatorname{W}\rvert^2&=\lvert\operatorname{W}^+\rvert^2+
			\lvert\operatorname{W}^-\rvert^2 \label{squarenormweyl}\\
			\lvert\operatorname{W}^{\pm}\rvert^2&=
			\dfrac{1}{2}\left[\lvert\operatorname{W}\rvert^2\pm
			\operatorname{tr}(\operatorname{W}\circ\operatorname{(\star W)})\right],
		\end{align}
		where $\operatorname{tr}(\operatorname{W}\circ\operatorname{(\star W)})=W_{ijkl}(\star W)_{ijkl}$.
		
		By \eqref{weylconst}, 
		\eqref{volform} and \eqref{levicivitasymbol}, we can 
		compute \eqref{weylplusminuscomp} for the hypersurface $(M,g)$:
		first, we compute \eqref{locstarweyl} by \eqref{weylconst} exploiting
		the skew-symmetry of \eqref{levicivitasymbol} as
		\begin{align} \label{hodgeweylhyper}
			(\star W)_{ijkl}&=\dfrac{1}{2}\mu_{ijrs}W_{rskl}=\\
			&=\mu_{ijrs}A_{rk}A_{sl}-\frac{H}{2}(\mu_{ijrl}A_{rk}-
			\mu_{ijrk}A_{rl})+\dfrac{1}{2}(\mu_{ijrl}A_{rk}^2
			-\mu_{ijrk}A_{rl}^2)+\dfrac{1}{6}(H^2-S)
			\mu_{ijkl}. \notag
		\end{align}
		By \eqref{hodgeweylhyper} we obtain the expression of 
		\eqref{weylplusminuscomp}:
		\small
		\begin{align} \label{weylplusminushyper}
			W_{ijkl}^{\pm}&=\dfrac{1}{2}[W_{ijkl}\pm(\star W)_{ijkl}]=\\
			&=\dfrac{1}{2}(A_{ik}A_{jl}-A_{il}A_{jk}\pm
			\mu_{ijrs}A_{rk}A_{sl})-
			\frac{H}{4}(A_{ik}\delta_{jl}-A_{il}\delta_{jk}+A_{jl}\delta_{ik}-
			A_{jk}\delta_{il}\pm\mu_{ijrl}A_{rk}\mp\mu_{ijrk}A_{rl})+\notag\\
			&+\dfrac{1}{4}(A^2_{ik}\delta_{jl}-A_{il}^2\delta_{jk}+
			A_{jl}^2\delta_{ik}-A_{jk}^2\delta_{il}	\pm\mu_{ijrl}A_{rk}^2\mp\mu_{ijrk}A_{rl}^2)+\notag\\
			&+\dfrac{1}{12}(H^2-S)(\delta_{ik}\delta_{jl}-
			\delta_{il}\delta_{jk}\pm\mu_{ijkl}). \notag
		\end{align}
		\normalsize
		Now, we want to compute $\lvert\operatorname{W}^{\pm}\rvert^2$: first,
		we recall that, given a $(0,2)$ symmetric tensor $T$, with components
		$T_{ij}$, 
		\begin{equation} \label{volformcontractsymm}
			\mu_{ijkl}T_{ij}=0 \quad \mbox{and} \quad \mu_{ijkl}T_{it}T_{tj}=0;
		\end{equation}
		indeed, the first equation in \eqref{volformcontractsymm} is
		trivial, while the second can be easily derived as follows
		\[
		\mu_{ijkl}T_{it}T_{tj}=\mu_{jikl}T_{jt}T_{ti}=
		\mu_{jikl}T_{ti}T_{jt}=\mu_{jikl}T_{it}T_{tj}=
		-\mu_{ijkl}T_{it}T_{tj}.
		\]
		By \eqref{symbolcontraction}, \eqref{weylplusminushyper}, 
		and \eqref{volformcontractsymm}, a long, yet straightforward, computation
		leads to 
		\begin{equation} \label{weylplusminusnormhyper}
			\lvert\operatorname{W}^{\pm}\rvert^2=
			\dfrac{7}{6}S^2+\dfrac{1}{6}H^4
			-2\lvert A^2\rvert^2+2H\operatorname{tr}A^3
			-\dfrac{4}{3}H^2S,
		\end{equation}
		where 
		\begin{equation} \label{htraces}
			\lvert A^2\rvert^2=
			A_{ij}^2A_{ij}^2, \quad \mbox{and} \quad
			\operatorname{tr}A^3=A^3_{ii}.
		\end{equation}
		By \eqref{squarenormweyl}, we finally obtain
		\begin{equation} \label{weylnormhyper}
			\lvert\operatorname{W}\rvert^2=
			\dfrac{7}{3}S^2+\dfrac{1}{3}H^4
			-4\lvert A^2\rvert^2+4H\operatorname{tr}A^3
			-\dfrac{8}{3}H^2S;
		\end{equation}
		note that, by \eqref{weylplusminusnormhyper} and \eqref{weylnormhyper},
		we have
		\begin{theorem} \label{equalweylnorms}
			Let $(M^4,g)$ be an oriented, isometrically immersed hypersurface in a space form $(N^5(c),g_N)$ with constant sectional curvature $c$. Then, at every $p\in M$,        \begin{equation}
				\label{weylequal}
				\lvert\operatorname{W}^+\rvert^2=
				\lvert\operatorname{W}^-\rvert^2=
				\dfrac{1}{2}\lvert\operatorname{W}\rvert^2.
			\end{equation}
		\end{theorem}
		This immediately leads to the following 
		\begin{cor} \label{halfconfmin}
			Let $(M^4,g)$ be an oriented, isometrically immersed hypersurface in a space form $(N^5(c),g_N)$ with constant sectional curvature $c$. Then $(M^4,g)$ is
			half conformally flat if and only if it is locally conformally flat.
		\end{cor}
		Now, we recall a fundamental topological identity.
		The splitting $\Lambda^2=\Lambda_+\oplus\Lambda_-$ induces
		a decomposition of the space of harmonic $2$-forms
		$H^2(M,\mathbb{R})$, which
		is isomorphic to the second de Rham cohomology group, \emph{via}
		the Hodge Theorem:
		\begin{equation} \label{cohomdeco}
			H^2(M,\mathbb{R})=H_+^2(M,\mathbb{R})\oplus H_-^2(M,\mathbb{R}),
		\end{equation}
		where $H_+^2(M,\mathbb{R})$ (resp., $H_-^2(M,\mathbb{R})$)
		is the space of \emph{self-dual} (resp., \emph{anti-self-dual})
		harmonic forms (recall \eqref{formdeco}). If $b_2$ denotes
		the second Betti number, i.e. the dimension of $H^2(M,\mathbb{R})$,
		we can also define $b_2^{\pm}$ as the dimensions of 
		$H_{\pm}^2(M,\mathbb{R})$ and obtain a fundamental invariant,
		which is the \emph{signature} $\tau(M)$ of the manifold $M$,
		defined as
		\begin{equation} \label{signature}
			\tau(M)=b_2^+-b_2^-.
		\end{equation} 
		As we did in the Introduction, we can define \eqref{signature} as the signature
		of the intersection form induced by the cup product on 
		$H^2(M,\mathbb{R})\times H^2(M,\mathbb{R})$. 
		The well-known Hirzebruch's signature
		formula allows us to express \eqref{signature} in terms
		of an integral involving curvature quantities: indeed,
		for every four-dimensional closed Riemannian manifold $(M,g)$, 
		we have the following integral formula (\eqref{hirze} in the Introduction)
		\begin{equation*}
			48\pi^2\tau(M)=\int_M\lvert\operatorname{W}^+\rvert^2
			-\lvert\operatorname{W}^-\rvert^2dV_g,
		\end{equation*}
		which, by \eqref{locstarweyl} and
		\eqref{squarenormweyl},
		can also be written as
		\[
		48\pi^2\tau(M)=\int_M
		\operatorname{tr(\operatorname{W}\circ (\star W))dV_g.
		}
		\]
		Since reversing the orientation swaps 
		$\Lambda_+$ and $\Lambda_-$ and, consequently, $\operatorname{W}^+$ and
		$\operatorname{W}^-$, by looking at the definition of the intersection form
		it is easy to see that 
		$\tau(M)$ changes sign under an orientation change: however, 
		$\lvert\tau(M)\rvert$ is a topological invariant. 
		
		From a topological viewpoint, it
		is possible to exploit some 
		properties of the Weyl tensor to
		draw some strong conclusions: for instance, \eqref{hirze} and \eqref{weylequal} allow us to derive the following result:
		\begin{cor} \label{signaturehyper}
			Let $(M^4,g)$ be a closed, oriented, isometrically immersed hypersurface in a space form $(N^5(c),g_N)$. Then, 
			$\tau(M)=0$. 
		\end{cor}
		\begin{rem} \label{pontrjagin}
			Corollary \ref{signaturehyper} is, in fact, a particular case
			of a more general result.  Indeed, in dimension four
			$\tau(M)=\frac{1}{3}p_1(M)$, where $p_1(M)=p_1(TM)$ is the first
			Pontrjagin class of the tangent bundle $TM$: it is know that, given a Whitney sum of vector bundles $E$ and $F$ over a smooth manifold $M$, the total Pontrjagin class $p(E\oplus F)$ is realized as the cup product of $p(E)$ and $p(F)$, which implies that $p_1(E\oplus F)=p_1(E)+p_1(F)$. 
			In our case, given a submanifold $M$ of a smooth manifold $N$, by the naturality of the Pontrjagin classes with respect to the pullback of bundles (see e.g. \cite{hirze}), we have
			\[
			p_1(TN)=p_1(TM)+p_1(TM^{\perp}).
			\]
			If $N$ and $M$ are orientable and $M$ has codimension $1$, $TM^{\perp}$ is a trivial line bundle, then its first Pontrjagin class vanishes: hence, $p_1(TN)=p_1(TM)$, which means that, if the first Pontrjagin class of $TN$ vanishes, then $p_1(TM)=0$. Therefore, we can say that, given an orientable smooth 5-manifold $N$ with vanishing first Pontrjagin class, every oriented hypersurface $M$ of $N$ must have $\tau(M)=0$: we highlight the fact that every 5-manifold admitting a Riemannian metric with constant sectional curvature has vanishing first Pontrjagin class (see also Remark \ref{remavez}). 
		\end{rem}
		Now recall that, by Poincaré duality,
		the Euler characteristic of a four-dimensional compact topological
		manifold $X$ is given by
		\begin{equation} \label{eulerclosed}
			\chi(X)=2b_0-2b_1+b_2.
		\end{equation}
		Note that $b_0=1$ if and only if $M$ is connected.
		By Corollary \ref{signaturehyper} and \eqref{eulerclosed}
		we can easily state the following
		\begin{cor} \label{eulercor}
			The Euler characteristic $\chi(M)$ of $M$ is an even integer. 
		\end{cor}
		\begin{proof}
			Since $\tau(M)=0$ by Corollary \ref{signaturehyper}, we have that,
			by \eqref{cohomdeco} and \eqref{signature},
			\[
			\operatorname{dim}H^2(M,\mathbb{R})=b_2=
			b_2^++b_2^-=2b_2^+.
			\]
			Therefore, \eqref{eulerclosed} becomes
			\[
			\chi(M)=2(b_0-b_1+b_2^+).
			\]
		\end{proof}
		Since we are dealing with 
		four-dimensional hypersurfaces in
		space forms, it
		is natural to consider an \emph{extrinsic version of the
			Chern--Gauss--Bonnet formula} for
		the Euler characteristic. First, let
		us define the \emph{traceless Ricci
			tensor} $\mathring{\operatorname{Ric}}=
		\operatorname{Ric}-\frac{R}{4}g$:
		by \eqref{ricconst} and
		\eqref{scalconst}, we immediately
		obtain
		\begin{equation} \label{riccisquare}
			\lvert\mathring{\operatorname{Ric}}\rvert^2=\lvert A^2\rvert^2-
			\dfrac{1}{4}S^2
			+\dfrac{3}{2}H^2S-2H\operatorname{tr}(A^3)-\dfrac{H^4}{4}.
		\end{equation}
		Recall that the usual Chern--Gauss--Bonnet
		formula reads as
		\begin{equation} \label{cherngaussbonnet}
			32\pi^2\chi(M)=\int_M\lvert\operatorname{W}\rvert^2-2\lvert\mathring{\operatorname{Ric}}\rvert^2+\dfrac{R^2}{6}dV_g:
		\end{equation}
		hence, exploiting 
		\eqref{scalconst}, \eqref{weylnormhyper} and \eqref{riccisquare}, we obtain
		the following (see also \cite{case2024gauss} for some analogy in the case of Poincaré-Einstein ambient spaces)
		\begin{lemma}[Extrinsic Chern--Gauss--Bonnet formula]
			Let $(M^4,g)$ be a closed, oriented, isometrically immersed hypersurface of a space form 
			$(N^5(c),g_N)$. Then, the following integral equality holds:
			\begin{equation} \label{cherngaussbonnethyper}
				32\pi^2\chi(M)=
				\int_M\left[
				3S^2-
				6\lvert A^2\rvert^2-
				\dfrac{16}{3}H^2S+H^4+8H\operatorname{tr}(A^3)+
				4c(6c-S+H^2)\right]dV_g.
			\end{equation}
			In particular, if $M$ is minimal,
			we have
			\begin{equation} \label{cherngaussbonnetminimal}
				32\pi^2\chi(M)=\int_M
				\left[3S^2-
				6\lvert A^2\rvert^2+
				4c(6c-S)\right]dV_g.
			\end{equation}
		\end{lemma}
		The study of the Weyl tensor
		in the context of immersed
		hypersurfaces might lead to 
		interesting characterization
		results: for instance, as we saw before,
		it is possible to identify
		$\operatorname{W}^{\pm}$ with symmetric, and
		hence diagonalizable,
		endomorphisms $\mathcal{W}^{\pm}$ of $\Lambda_{\pm}$. Furthermore, if we choose an orthonormal Darboux frame around $p\in M$ such that $A$ is in diagonal form at $p$ (see \eqref{diagonalsecond}), it is easy to realize that the endomorphisms $\mathcal{W}^{\pm}$ are in diagonal form at $p$. In virtue of \eqref{weylplusminushyper}, we can see that
		\begin{equation} \label{weyleigen}
			W_{ijij}^{\pm}=
			\dfrac{1}{4}(\lambda_i+\lambda_j)\left(\lambda_i+\lambda_j-H\right)+
			\dfrac{1}{12}\left(H^2-S\right), \quad i<j;
		\end{equation}
		note that, by \eqref{hodgeweylplusmincomp},
		\[
		W_{1212}^{\pm}=\pm W_{1234}^{\pm}=W_{3434}^{\pm}, \quad
		W_{1313}^{\pm}=\pm W_{1342}^{\pm}=W_{4242}^{\pm}, \quad
		W_{1414}^{\pm}=\pm W_{1423}^{\pm}=W_{2323}^{\pm}.
		\]
		Hence, it is easy to observe that the eigenvalues of $\mathcal{W}^{\pm}$ are 
		$2W_{1212}^{\pm}$, $2W_{1313}^{\pm}$ and $2W_{1414}^{\pm}$ with respect
		to a local orthonormal frame which diagonalizes $\mathcal{W}^{\pm}$:
		moreover, for every $i\neq j$, $W_{ijij}^+=W_{ijij}^-$ with respect to this frame.
		
		First, we rewrite a particular case of the well-known characterization result due to
		Nishikawa and Maeda 
		(\cite{NiskikawaMaeda1974}):
		\begin{proposition} \label{nishi}
			Let $(M^4,g)$ be an oriented, isometrically immersed hypersurface in a space form $(N^5(c),g_N)$.
			Then, the (anti-)self-dual Weyl tensor of $(M,g)$ vanishes at $p\in M$ if and only if at least
			three principal curvatures
			are equal at $p$.
		\end{proposition}
		\begin{proof}
			Assume that three
			principal curvatures are
			equal at $p\in M$: for instance, without loss of
			generality we can set $\lambda:=\lambda_1=\lambda_2=\lambda_3$ and $\mu=\lambda_4$. 
			In this case, \eqref{weylplusminusnormhyper}
			becomes
			\begin{align*}
				\lvert\operatorname{W}^{\pm}\rvert^2&=
				\dfrac{7}{6}(3\lambda^2+\mu^2)^2+
				\dfrac{1}{6}(3\lambda+\mu)^4
				-2(3\lambda^4+\mu^4)+\\
				&+2(3\lambda+\mu)(3\lambda^3+\mu^3)-
				\dfrac{4}{3}(3\lambda+\mu)^2
				(3\lambda^2+\mu^2)=0.
			\end{align*}
			Conversely, let 
			$\operatorname{W}^+\equiv 0$
			at $p\in M$: since,
			by hypothesis, 
			$W_{ijij}^+=0$ for every $i,j$,
			choosing a frame which diagonalizes $\mathcal{W}^+$, by \eqref{weyleigen} we obtain
			\[
			(\lambda_1+\lambda_2)(\lambda_1+\lambda_2-H)=
			(\lambda_1+\lambda_3)(\lambda_1+\lambda_3-H)=
			\dfrac{1}{3}(S-H^2)
			\]
			The first equality can be rewritten as
			\[
			(\lambda_2-\lambda_3)[2\lambda_1+\lambda_2+\lambda_3-
			H]=0.
			\]
			If $\lambda_2\neq\lambda_3$ at $p$,
			then 
			\[
			0=2\lambda_1+\lambda_2+\lambda_3-
			H=\lambda_1-\lambda_4,
			\]
			i.e. $\lambda_1=\lambda_4$. Hence
			we obtain that either $\lambda_2=\lambda_3$ or $\lambda_1=\lambda_4$: without loss
			of generality, we may assume that 
			$\lambda_2=\lambda_3$. Then,
			using the definition of $H$, we obtain
			\[
			-(\lambda_1+\lambda_2)(\lambda_2+\lambda_4)=\dfrac{1}{3}
			[\lambda_1^2+2\lambda_2^2+
			\lambda_4^2-(\lambda_1+2\lambda_2+\lambda_4)^2],
			\]
			i.e.
			\[
			\lambda_2^2-\lambda_1\lambda_2+
			\lambda_1\lambda_4-\lambda_2\lambda_4=0
			\Longrightarrow 
			(\lambda_1-\lambda_2)(\lambda_2-\lambda_4)=0
			\]
			and this concludes the proof.
		\end{proof}
		\begin{rem}
			We recall that, by the work of Do Carmo and Dajczer, a minimal hypersurface in a complete, simply connected space form is locally conformally flat if and only if it is contained in a catenoid \cite{docarmodaj}. 
		\end{rem}
		
		There exists a connection between the eigenvalues of $\mathcal{W}^{\pm}$ and the principal curvatures of an immersed hypersurface: namely, we can prove the following
		\pagebreak
		\begin{theorem} \label{eigentheorem}
			Let $(M^4,g)$ be an oriented
			hypersurface isometrically immersed in a space form $(N^5(c),g_N)$:
			\begin{enumerate}
				\item \label{firststateWeyl} if $m$ is the number
				of distinct principal curvatures 
				and $w$ is the number of distinct
				eigenvalues of $\mathcal{W}^+$, then,
				at $p\in M$
				\begin{itemize}
					\item $w=3$
					if and only if $m=4$;
					\item $w=2$ if and only if 
					$m=3$ or there are two distinct principal curvatures, both with multiplicity two (in this case, $m=2$);
					\item $w=1$ if and only if 
					at least three principal curvatures are equal (and, hence, $m\leq 2$). In this case,
					$\operatorname{W}^+\equiv 0$
					at $p$;
				\end{itemize}
				\item \label{secondstateWeyl} if the scalar curvature $R$ and
				the mean curvature of $(M^4,g)$ are
				constant, $(M^4,g)$ is
				isoparametric if and only if
				the eigenvalues of 
				$\mathcal{W}^+$ are
				constant functions. 
			\end{enumerate}
			The same results hold for $\mathcal{W}^-$.
		\end{theorem}
		\begin{proof}
			First, it is obvious that interchanging the roles of $\mathcal{W}^+$ and $\mathcal{W}^-$ does not change the statement, since, at every point, they have the same eigenvalues by \eqref{weyleigen}. 
			\begin{enumerate}
				\item The last claim is
				immediate to show, since it is
				a direct consequence of the
				computations made in
				the proof of Proposition 
				\ref{nishi}: recall that,
				since $\operatorname{W}^+$ is
				a traceless tensor, if
				all the eigenvalues of $\mathcal{W}^+$ are all equal they
				have to be $0$, which means that
				$\operatorname{W}^+\equiv 0$ at $p$.
				
				Now,
				assume that $w=2$ at $p\in M$:
				without loss of generality,
				we might suppose that 
				$W_{1212}^+=W_{1313}^+$.
				Then, \eqref{weyleigen} implies that
				\[
				(\lambda_1+\lambda_2)(\lambda_1+\lambda_2-H)=(\lambda_1+\lambda_3)(\lambda_1+\lambda_3-H).
				\]
				Following the same line
				of reasoning of Proposition
				\ref{nishi}, we obtain that
				either $\lambda_2=\lambda_3$ or
				$\lambda_1=\lambda_4$: therefore,
				$m\leq 3$. Since $w=2$, $\operatorname{W}^+$ does not
				identically vanish at $p$, which means
				that, if $m=2$ at $p$, then both distinct principal curvatures have multiplicity two, since there cannot be three equal principal curvatures. 
				
				Conversely, by \eqref{weyleigen}
				it is obvious that, if $m=3$, then
				$w=2$: furthermore, if there are two distinct principal curvatures with multiplicity two, 
				then $w>1$, otherwise we 
				would have that $\operatorname{W}^+
				\equiv 0$ at $p$.
				
				\noindent
				Finally, since all the other cases
				are covered, we conclude that 
				$w=3$ if and only if $m=4$.
				\item
				It is
				immediate to see that, if 
				$(M^4,g)$ is isoparametric, then
				the eigenvalues of $\mathcal{W}^+$ are constant functions. Conversely, we first note that, since $w$ is constant on $M$, $w=1$ if and only if $(M^4,g)$ is locally conformally flat: furthermore, since $S$ is constant on $M$, this implies that $M$ is isoparametric. Now, let us assume that $w=2$, which means that, at every $p\in M$, there are at most three distinct principal curvatures, i.e. $m\leq 3$. Without loss of generality, assume that $\lambda_1=\lambda_2$ at $p\in M$ and denote $W_{1212}^+=C$: a straightforward computation shows that, up to exchanging the roles of $\lambda_3$ and $\lambda_4$,
				\begin{align*}
					\lambda_1&=\lambda_2=\dfrac{H\pm \sqrt{H^2+8C}}{4}\\ \lambda_3&=
					\dfrac{1}{4}\left(H\mp\sqrt{H^2+8C}+D\right)\\ 
					\lambda_4&=
					\dfrac{1}{4}\left(H\mp\sqrt{H^2+8C}-D\right),
				\end{align*}
				where $D:=\sqrt{8S-4H^2-16C}$.
				This immediately shows that, if $m=3$ on $M$, then $M$ is isoparametric: however, this is impossible, by Cartan's classification of isoparametric hypersurfaces with $m=3$. Hence there exists $p\in M$ such that $m=2$: however, by the previous computations, since we are assuming that $w=2$ (i.e., $\operatorname{W}$ does not vanish at any point of $M$), we would obtain that $D=0$ at $p\in M$, i.e.
				\[
				8S-4H^2-16C=0,
				\]
				which holds, in fact, on $M$, since all the quantities involved are constant on $M$. This line of reasoning allows us to show that $m=2$ on $M$: since $S$ is constant, this immediately implies that $M$ is isoparametric. 
				
				Finally, let $w=3$, which means that $m=4$ on $M$, i.e. that $\lambda_i\neq\lambda_j$
				for $i\neq j$ at every $p\in M$: since
				the eigenvalues of $A$ are simple, given $p\in M$ the second fundamental form can be diagonalized in an open neighborhood $U_p$ of $p$.
				By the
				fact that $H^2$ and $S$ are constant functions, we can differentiate 
				\eqref{weyleigen} to get
				\[
				0=dW_{ijij}^+=\dfrac{1}{4}
				d(\lambda_i+\lambda_j)\left[
				2(\lambda_i+\lambda_j)-
				H\right].
				\]
				Let us assume that
				$d(\lambda_i+\lambda_j)\neq 0$ at $p$
				for some $i\neq j$:
				then, by local smoothness of the 
				principal curvatures, this holds
				in an open neighborhood $U'_p\subset U_p$.  
				However, this means that
				\[
				\lambda_i+\lambda_j=\dfrac{H}{2} \quad\mbox{in } U'_p,
				\]
				which is impossible, since $H$ is
				constant. Hence, at $p$, 
				\[
				d(\lambda_i+\lambda_j)=0
				\]
				for every $i,j=1,..,4$, $i\neq j$.
				Now, by the fact that $H$ is constant,
				we obtain
				\[
				0=d(\lambda_1+\lambda_2)=
				d(\lambda_1+\lambda_3) \Longrightarrow
				0=d(2\lambda_1+\lambda_2+\lambda_3)
				\Longrightarrow 0=d(\lambda_1-\lambda_4);
				\]
				since $d(\lambda_1+\lambda_4)=0$, 
				we conclude that $d\lambda_1=d\lambda_4=0$ and, hence,
				$d\lambda_i=0$ for every $i$. 
			\end{enumerate}
		\end{proof}

		If $\iota:(M,g)\longrightarrow (N,g_N)$ is a minimal isometric immersion, we have the following expressions for \eqref{weylconst}, \eqref{weylplusminushyper} and \eqref{weylplusminusnormhyper}:
		
		\begin{align} \label{weylminimal}
			W_{ijkl}&=A_{ik}A_{jl}-A_{il}A_{jk}+
			\dfrac{1}{2}(A_{ik}^2\delta_{jl}-A_{il}^2\delta_{jk}
			+A_{jl}^2\delta_{ik}-A_{jk}^2\delta_{il})-\dfrac{1}{6}S
			(\delta_{ik}\delta_{jl}-\delta_{il}\delta_{jk});
		\end{align}
		\begin{align} \label{weylplusminusminimal}
			W_{ijkl}^{\pm}&=\dfrac{1}{2}(A_{ik}A_{jl}-A_{il}A_{jk}\pm
			\mu_{ijrs}A_{rk}A_{sl})+\\
			&+\dfrac{1}{4}(A_{ik}^2\delta_{jl}-A_{il}^2\delta_{jk}+
			A_{jl}^2\delta_{ik}-A_{jk}^2\delta_{il}\pm\mu_{ijrl}A_{rk}^2\mp\mu_{ijrk}A_{rl}^2)+\notag\\
			&-\dfrac{1}{12}S(\delta_{ik}\delta_{jl}-
			\delta_{il}\delta_{jk}\pm\mu_{ijkl}); \notag
		\end{align}
		\begin{equation} \label{weylplusminusnormminimal}
			\lvert\operatorname{W}^{\pm}\rvert^2=
			\dfrac{7}{6}S^2
			-2\lvert A^2\rvert^2.
		\end{equation}
		The Weyl tensor can also be written by means of the so-called \emph{Fialkow tensor}, defined as (see e.g \cite{case2023sharp, Fialkow1944})
		\begin{equation} \label{fialkow}
			\operatorname{F}=\dfrac{1}{2}\left(A^2-\dfrac{1}{6}Sg\right):
		\end{equation}
		by \eqref{fialkow}, it is possible to show that
		\begin{equation} \label{weylfialkowglob}
			\operatorname{W}=\dfrac{1}{2}A\KN A + \operatorname{F}\KN  g,
		\end{equation}
		where $\KN$ is the Kulkarni-Nomizu product, which, given two $(0,2)$-tensor fields $U$ and $T$, is locally defined as
		\begin{equation}
			(U\KN T)_{ijkl}= U_{ik}T_{jl}+U_{jl}T_{ik}-U_{il}T_{jk}-U_{jk}T_{il}.
		\end{equation}
		Locally, by \eqref{fialkow} we
		can rewrite \eqref{weylminimal} as
		\begin{equation} \label{weylfialkow}
			W_{ijkl}=A_{ik}A_{jl}-A_{il}A_{jk}+
			F_{ik}\delta_{jl}-F_{il}\delta_{jk}+F_{jl}\delta_{ik}-F_{jk}\delta_{il}
		\end{equation}
		and, consequently, \eqref{weylplusminusminimal} as
		\begin{equation}
			W_{ijkl}^{\pm}=\dfrac{1}{2}
			(A_{ik}A_{jl}-A_{il}A_{jk}+
			F_{ik}\delta_{jl}-F_{il}\delta_{jk}+F_{jl}\delta_{ik}-F_{jk}\delta_{il})+\dfrac{1}{2}\mu_{ijrs}
			(A_{rk}A_{sl}+
			F_{rk}\delta_{sl}+F_{sl}\delta_{rk}).
		\end{equation}
		By the formulas 
		\eqref{weylnormhyper} and \eqref{riccisquare}, we can
		prove two sharp inequalities which will be of crucial importance throughout the paper: namely, it is immediate show the following
		\begin{proposition}
			Let $(M^4,g)$ be a minimal, isometrically immersed hypersurface in a space form $(N^5(c),g_N)$. Then, at every $p\in M$, we have 
			\begin{equation} \label{ineq}
				\dfrac{1}{4}S^2
				\leq \lvert A^2\rvert^2\leq \dfrac{7}{12}S^2;
			\end{equation}
			moreover, at $p$, equality holds in the left-hand side if and only if $\mathring{\operatorname{Ric}}\equiv 0$ and equality holds in the right-hand side if and only if $\operatorname{W}\equiv 0$.
		\end{proposition}
		\begin{rem} \label{inequalities}
			The right-hand side of \eqref{ineq} is a particular case of a more general result proven in \cite{case2023sharp}, which states that, given a symmetric, trace-free
			endomorphism $A$ of a $n$-dimensional inner product space, then 
			\begin{equation*} 
				\dfrac{n^2-3n+3}{n(n-1)}S^2\geq \lvert A^2\rvert^2,
			\end{equation*}
			with equality if and only if $A$ admits an eigenspace of dimension $n-1$. In fact, if $A$ is the second fundamental form of a minimal immersion into a $(n+1)$-dimensional space form, then the squared norm of the Weyl tensor of the induced metric on the hypersurface is given by
			\begin{equation} \label{weylminimalgeneral}
				\lvert\operatorname{W}\rvert^2=\dfrac{2(n^2-3n+3)}{(n-1)(n-2)}S^2-\dfrac{2n}{n-2}\lvert A^2\rvert^2,
			\end{equation}
			which is non-negative and, therefore, yields
			\begin{equation} \label{weylminimalgeneral2}
				\dfrac{n^2-3n+3}{n(n-1)}S^2-\lvert A^2\rvert^2\geq 0.
			\end{equation}
			On the other hand, in higher dimensions, the left-hand side of \eqref{ineq} becomes
			\begin{equation} \label{tracriccmingen}
				\lvert A^2\rvert^2\geq \dfrac{1}{n}S^2;
			\end{equation}
			we also point out that, since \eqref{tracriccmingen} is obtained by applying the Cauchy-Schwarz inequality and $A$ is a symmetric tensor, equality holds at $p\in M$ if and only if the principal curvatures 
			$\lambda_i$ satisfy
			\[
			\lambda_i^2=\lambda_j^2, \quad
			\mbox{for every } i,j,
			\]
			which means that, if $n$ is odd, $A\equiv 0$ at $p$, while, if $n$ is even, either $A\equiv 0$ or there are exactly two distinct principal curvatures $\lambda_1$ and $\lambda_2$ at $p$, both with multiplicity $n/2$, such that $\lambda_1=-\lambda_2$. If the latter case holds at every point $p$, then by \eqref{ricconst} we conclude that $(M^n,g)$ is a non-totally geodesic Einstein hypersurface (see \cite{Fialkow1938, lawson, ryan} for classification results).
		\end{rem}
		
		We end this section with a useful Proposition, which will be used throughout the paper. As said before, the classification of minimal hypersurfaces with $S\equiv n$ in $\mathbb{S}^{n+1}$ was completed by Chern, do Carmo, Kobayashi \cite{cherndocarmokob} and Lawson \cite{lawson}: here, we include a local version of this result, considering a more general space form with positive curvature as the ambient space. Namely, we prove the following
		
		\begin{proposition} \label{localisometryprop}
			Let $(M^n,g)$ be a closed, minimally immersed hypersurface in a space form $(N^{n+1}(1),g_N)$. If $S\equiv n$, then $(M^n,g)$ is locally isometric to one of the Clifford hypersurfaces
			\[
			\mathbb{S}^k\left(\sqrt{\dfrac{k}{n}}\right)\times\mathbb{S}^{n-k} \left(\sqrt{\dfrac{n-k}{n}}\right),
			\]
			endowed with the standard product metric, for some $k=1,...,\lfloor n/2\rfloor$ 
			(here $\mathbb{S}^N(r)$ indicates the standard $N$-dimensional sphere of radius $r$). 
		\end{proposition}
		\begin{proof}
			Since $S\equiv n$ on $M$, by \eqref{simonsidint} we immediately obtain that the second fundamental form of the immersion is parallel,
			i.e. $\nabla A\equiv 0$ on $M$: this also implies that $M$ must be isoparametric. By a result due to M\"{u}nzner
			\cite{munzner} (also mentioned in the Introduction), recalling that $m$ is the number of distinct principal curvatures, we have that
			\[
			n=S=(m-1)n \Longrightarrow 
			m=2.
			\]
			Hence, there are two distinct principal curvatures $\lambda_1$ and $\lambda_2$ such that $\lambda_1$ has multiplicity $k$ and $\lambda_2$ has multiplicity $n-k$: since $S\equiv n$, $M$ is not totally geodesic, hence $k\geq 1$. 
			By the fact that $M$ is minimal and $S\equiv n$, we easily obtain that
			\[
			\lambda_1=\pm\sqrt{\dfrac{n-k}{k}}, \quad \mbox{ and } \quad 
			\lambda_2=\mp\sqrt{\dfrac{k}{n-k}}.
			\]
			
			Now, since $S\neq 0$, $A$ is not proportional to $g$: therefore, since $A$ is a parallel, symmetric $(0,2)$-tensor, we can apply a splitting result due to Eisenhart \cite{eisen} (or, equivalently, an analogous of the De Rham's splitting Theorem for symmetric, parallel $(0,2)$-tensors, see e.g. \cite{boubel}) to conclude that 
			$(M^n,g)$ must be locally isometric to a Riemannian product, where the dimension of the factors is induced by the dimension of the eigenspaces of $A$. Since $m=2$, we immediately obtain that $M$ is locally isometric to $(M_1^k\times M_2^{n-k},g_1\oplus g_2)$, where $(M_1^k,g_1)$ (resp. $(M_2^{n-k},g_2))$ is a Riemannian manifold of dimension $k$ (resp. of dimension $n-k$): moreover, if we use the index notation $p,q,r,...=1,...,k$ and $a,b,c,...=k+1,...,n$, by recalling the local expression of the curvature of a product metric and using \eqref{riemconst}, with respect to a frame that diagonalizes $A$ we obtain
			\begin{align*}
				{}^{M_1}R_{pqrs}&=R_{pqrs}=
				(\delta_{pr}\delta_{qs}-\delta_{ps}\delta_{qr})+A_{pr}A_{qs}-A_{ps}A_{qr}=
				\dfrac{n}{k}(\delta_{pr}\delta_{qs}-\delta_{ps}\delta_{qr})\\
				{}^{M_2}R_{abcd}&=R_{abcd}=
				(\delta_{ac}\delta_{bd}-\delta_{ad}\delta_{bc})+A_{ac}A_{bd}-A_{ad}A_{bc}=
				\dfrac{n}{n-k}(\delta_{ac}\delta_{bd}-\delta_{ad}\delta_{bc}),
			\end{align*}
			where ${}^{M_i}R_{\alpha\beta\gamma\delta}$ denote the components of the Riemann curvature tensor of $M_i$. Therefore, $M_1$ (resp. $M_2$) is a space form of dimension $k$ (resp. $n-k$) with constant sectional curvature equal to $n/k$ (resp. $n/(n-k)$): since positively curved space forms are quotients by actions of the isometry group of the sphere endowed with the standard metric, the claim follows by recalling that the quotient maps are local isometries. 
		\end{proof}
		
		Since we mainly work with four-dimensional hypersurfaces, by Proposition \ref{nishi} and the discussion at the end of Remark \ref{inequalities}, we can also write the following
		\begin{cor} \label{localisometrycor}
			Let $(M^4,g)$ be a closed, minimally immersed hypersurface in a space form $(N^{5}(1),g_N)$. If $S\equiv 4$ on $M$, then one of the following holds:
			\begin{enumerate}
				\item $(M^4,g)$ is locally isometric to \\
				$\mathbb{S}^1(1/\sqrt{2})\times\mathbb{S}^3(\sqrt{3}/2)$
				and hence locally conformally flat.
				\item $(M^4,g)$ is locally isometric to $\mathbb{S}^2(1/\sqrt{2})\times\mathbb{S}^2(1/\sqrt{2})$
				and hence Einstein.
			\end{enumerate}
		\end{cor}
		\begin{proof}
			By Proposition \ref{localisometryprop}, we immediately obtain the local isometry.  
			\begin{enumerate}
				\item If $(M^4,g)$ is locally isometric to $\mathbb{S}^1\times\mathbb{S}^3$, then by the formulas for the principal curvatures appearing in the proof of Proposition \ref{localisometryprop}, we obtain that there are three equal principal curvatures at every point of $M$: by Proposition \ref{nishi}, this implies the local conformal flatness.
				\item Analogously, if $(M^4,g)$ is locally isometric to $\mathbb{S}^2\times\mathbb{S}^2$,
				there are exactly two distinct principal curvatures, both with multiplicity two, at every point of $M$: by Remark \ref{inequalities}, the claim is proven. 
			\end{enumerate}
		\end{proof}
		
		\begin{rem}
			Following Lawson \cite{lawson}, if $N^{n+1}$ is the standard sphere $\mathbb{S}^{n+1}(1)$, since we are assuming that the hypersurface is closed, the isometry must be global.
		\end{rem}
		We point out that, if the ambient space 
		$(N^5,g_n)$ is just locally conformally flat, for every point $p\in M\subset N$ there exist an open neighborhood $p\in U\subset N$ and a smooth function $u\in C^{\infty}(U)$ such that $(U,e^{2u}g_N)$ is a flat open submanifold of $N$: by restricting the conformal factor $u$ to a smooth function $\bar{u}$ on
		$M\cap U$, we obtain that, locally, the induced metric on $M\cap U$ is $\bar{g}:=e^{2\bar{u}}g$. Therefore, since a local Darboux frame for $\bar{g}$ is the rescaling of the one chosen for $g$, we obtain that the local components of the (anti-)self-dual Weyl tensor of $\bar{g}$ are expressed by the formula \eqref{weylplusminushyper}, where all the geometric quantities are now with respect to $\bar{g}$. Furthermore,
		the number $m$ of distinct eigenvalues of $A$ is invariant
		under conformal changes of the ambient metric, since the
		traceless second fundamental form $\mathring{A}=A-(H/n)g$ is
		conformally covariant, i.e., 
		\[
		\overline{\mathring{A}}=e^{\bar{u}}\mathring{A},
		\]
		where $\overline{\mathring{A}}$ is the traceless second fundamental form with respect to $\bar{g}$.
		Hence, we can perform some of the previous local computations regarding the Weyl tensor and then exploit its conformal properties to extend some of our results to the case where the ambient space is only locally conformally flat instead of a space form. More precisely, we have the following
		\begin{meta} \label{metathm}
			Theorem \ref{equalweylnorms}, Corollary \ref{halfconfmin}, Corollary \ref{signaturehyper}, Corollary \ref{eulercor}, Proposition \ref{nishi} and [Theorem \ref{eigentheorem},\ref{firststateWeyl}] hold if $(N^5,g_N)$ is a locally conformally flat manifold (not necessarily a space form). 
			In particular, Theorem \ref{equalweylnormslcf}, Corollary \ref{signaturehyperlcf} and Theorem \ref{eigentheoremlcf} hold. 
		\end{meta}
		\begin{rem} \label{remavez}
			Taking into account Remark \ref{pontrjagin}, we recall that Avez showed that every smooth manifold admitting a locally conformally flat metric has vanishing
			Pontrjagin classes \cite{avez}, which means that every hypersurface of such manifold has zero signature and this is an alternative proof of the fact that Corollary \ref{signaturehyperlcf} holds. However, in Theorem \ref{equalweylnormslcf} we show that $\lvert\operatorname{W}^+\rvert=\lvert\operatorname{W}^-\rvert$: to the best of our knowledge, the vanishing of the signature does not guarantee, \emph{a priori}, the existence of such a metric, while, on the other hand, every compact four-manifold admits a Riemannian metric such that $\lvert\operatorname{W}^++t\operatorname{W}^-\rvert\neq 0$ everywhere for every $t\in\mathbb{R}$ \cite{cdmaubin}.  
		\end{rem}

		\section{Topological bounds for the second fundamental form} \label{topbounds}
		The main aim of this section is to
		obtain sharp topological results involving the Euler characteristic of four-dimensional minimal hypersurfaces in space forms: in particular, we are interested in integral inequalities for the Euler characteristic in terms of the Weyl tensor and the squared norm of the second fundamental form $S$. Consequently, we obtain a lower bound for $S$ in terms of the Euler characteristic, possibly relating our analysis to the weak version of the Chern conjecture in a particular case. 
		
		First, we recall that, given a closed Riemannian manifold $(M,g)$ of dimension $n$, the \emph{Yamabe constant} $Y(M,[g])$ of the conformal class $[g]$ is defined as the infimum of the normalized Einstein--Hilbert functional in $[g]$, i.e.
		\begin{equation*} 
			Y(M,[g])=\inf_{g'\in [g]}\mathrm{Vol}_{g'}(M)^{-\frac{n-2}{n}}\int_MR_{g'}\,dV_{g'}:
		\end{equation*} 
		We refer the reader to the Introduction for some historical background on the importance of this conformal invariant in geometric analysis. Here, we provide a topological result concerning minimal hypersurfaces in $5$-dimensional space forms such that the induced metrics have positive Yamabe constant (Theorem \ref{boundyamabe}). 
		\begin{rem}
			We highlight the fact that, if $c\in\{0,-1\}$, there exist no closed minimal hypersurfaces if the ambient space is simply connected: hence, in Theorem \ref{boundyamabe} and throughout the rest of the paper, when $c$ is allowed to be different from $1$, we will assume that the ambient space is a non simply connected space form with vanishing or negative constant sectional curvature.
		\end{rem}

		\begin{proof}[Proof of Theorem \ref{boundyamabe}]
			First, the previous Remark allows us to assume that the ambient space is not simply connected if $c\in\{0,-1\}$. By the definition of the Yamabe constant $Y(M^4,[g])$ it is clear that minimal hypersurfaces in space forms with $c\in\{0,-1\}$ cannot have positive Yamabe constant: indeed, by \eqref{scalconst}, such a minimal hypersurface would have non-positive scalar curvature, which means that the infimum of the Einstein--Hilbert action in the conformal class $[g]$ must be non-positive. Then, by a well-known result due to Chang, Gursky and Yang \cite{chang2003conformally}, if the Yamabe constant is positive and 
			\[
			\int_M\lvert\operatorname{W}\rvert^2dV_g<16\pi^2\chi(M),
			\]
			$M$ has to be diffeomorphic to $\mathbb{S}^4$ (the result by Chang, Gursky and Yang states that $M$ might also be diffeomorphic to $\mathbb{RP}^4$, but this cannot be the case in our setting, since
			$\mathbb{RP}^4$ is not orientable). Hence, we assume that
			\begin{equation} \label{lowerboundweyl}
				\int_M |\operatorname{W}|^2dV_g\geq 16\pi^2\chi(M):
			\end{equation}
			this implies that
			\begin{equation*}
				16\pi^2\chi(M)+4\int_M|A^2|^2dV_g\leq \frac{7}{3}\int_M S^2dV_g.
			\end{equation*}
			Using the left-hand side of \eqref{ineq},
			we obtain
			\begin{align*}
				16\pi^2\chi(M)+\int_M S^2dV_g\leq \frac{7}{3}\int_M S^2dV_g,
			\end{align*}
			that is, 
			\begin{equation*}
				16\pi^2\chi(M)\leq\frac{4}{3}\int S^2dV_g,
			\end{equation*}
			which gives 
			\begin{equation*}
				4\pi^2\chi(M)\leq\frac{1}{3}\int_M S^2dV_g.
			\end{equation*}
			If we assume that equality holds in the
			previous equation, by \eqref{ineq} we would get that $(M^4,g)$ is an Einstein manifold and that
			\begin{equation} \label{equalitygursky}
				\int_M\lvert\operatorname{W}\rvert^2dV_g=
				16\pi^2\chi(M),
			\end{equation}
			which is impossible: indeed, if this were the case, $(M^4,g)$ would have to be conformally equivalent to $\mathbb{CP}^2$ with the Fubini--Study metric \cite{chang2003conformally}, which has non-zero signature and the contradiction follows from Corollary \ref{signaturehyper}. Moreover, we observe that the other case described in \cite{chang2003conformally} cannot be realized, since $(M^4,g)$ would be locally conformally flat and, therefore, totally geodesic, contradicting the validity of \eqref{equalitygursky},
			since the only totally geodesic hypersurfaces of $N^5(1)$ are
			space forms with positive scalar curvature and, therefore, 
			positive Euler characteristic by \eqref{cherngaussbonnetminimal}.
			
			Now, using again \eqref{cherngaussbonnetminimal}
			and \eqref{ineq} we get
			\begin{align*}
				4\pi^2\chi(M)&=\int_M\left(\dfrac{3}{8}S^2-\dfrac{3}{4}\lvert A^2\rvert^2 +3-\dfrac{1}{2}S\right)dV_g\geq\\
				&\geq\int_M\left(-\dfrac{1}{16}S^2+3-S\right)dV_g=\dfrac{1}{16}\int_M
				(4-S)(12+S)dV_g.
			\end{align*}
			Finally, it is clear that the equality in the previous inequality holds if and only if $(M^4,g)$ is locally conformally flat and, by \eqref{lowerboundweyl}, we obtain that $\chi(M)\leq 0$.
		\end{proof}
		\begin{rem} \label{remarkeulerbound}
			We point out that \eqref{eulerbound}
			is only interesting when $S$ is not bounded above by $x$, where
			\[
			x:=\dfrac{12}{19}(2\sqrt{5}+1):
			\]
			however, if $S<x$ on $M$, then obviously, by Simons' identity \eqref{simonsid}, $(M^4,g)$ must be a totally geodesic space form in $(N^5(1),g_N)$, since $x<4$.
			
			Furthermore, observe that the left-hand side of \eqref{eulerbound} can be obtained without any hypothesis on 
			the sign of the Yamabe constant: in that case, equality would hold if and only if the hypersurface is locally conformally flat, without the conclusion on the sign of $\chi(M)$.
			However, if we impose that $S$ is constant, by Corollary \ref{localisometrycor} we immediately obtain that the hypersurface is locally isometric to
			$\mathbb{S}^1\times\mathbb{S}^3$ and, therefore, $\chi(M)=0$ by means of the equality in the left-hand side of \eqref{eulerbound}.
		\end{rem}
		
		In the context of minimal immersions into space forms,
		we can find sharp topological lower bounds for the Weyl functional,
		inspired by the work due to Gursky
		\cite{GurskyAnnals, Gurskyharmonicweyl} (namely, Theorem \ref{topbound}, Corollary \ref{rigiditytopbound}
		and Corollary \ref{corbound} in
		the Introduction): however,
		in this special case we do not need to impose any condition on
		the sign of the Yamabe constant of the conformal class and we do not
		need to require the existence of a self-dual harmonic form.
		
		\begin{proof}[Proof of Theorem \ref{topbound} and Corollary \ref{rigiditytopbound}] 
			First, if the Euler characteristic is negative, both claims are trivial, hence we might assume that $\chi(M)\geq 0$. Then, we observe that, if $(M^4,g)$ is locally conformally flat, \eqref{boundeucl} does not hold unless $\chi(M)<0$: moreover, if $c=1$, since the scalar curvature is constant, the hypersurface is locally conformally flat if and only if it is totally geodesic (and, hence, it does not satisfy \eqref{boundsphere} by \eqref{cherngaussbonnetminimal}), or it is locally isometric to $\mathbb{S}^1\times \mathbb{S}^3$ and hence equality in \eqref{boundsphere}
			holds (e.g. by the last comment of Remark \ref{remarkeulerbound}). Now, suppose that $(M^4,g)$ is not locally conformally flat: this also allows us to assume that $\chi(M)>0$, since, if $\chi(M)=0$, both \eqref{boundeucl} and \eqref{boundsphere} are trivial. We want to make use of the Chern--Gauss--Bonnet formula \eqref{cherngaussbonnet}. 
			Let $D>0$ such that
			\[
			\int_M\lvert\operatorname{W}\rvert^2dV_g= D\pi^2\chi(M);
			\]
			by the Chern--Gauss--Bonnet formula, we have
			\[
			\dfrac{32}{D}\int_M\lvert\operatorname{W}\rvert^2dV_g=
			\dfrac{32}{D}\cdot D\pi^2\chi(M)=
			\int_M\left(\lvert\operatorname{W}\rvert^2
			-2\lvert\mathring{\operatorname{Ric}}\rvert^2
			+\dfrac{R^2}{6}\right)dV_g,
			\]
			which, by \eqref{scalconst}, \eqref{weylnormhyper} and \eqref{riccisquare} 
			can be rewritten as 
			\[
			\left(\dfrac{32}{D} -1\right)
			\int_M\left(\dfrac{7}{3}S^2-
			4\lvert A^2\rvert^2\right)dV_g=
			\int_M\left(24c^2-4cS+
			\dfrac{2}{3}S^2-2\lvert A^2\rvert^2\right)dV_g,
			\]
			i.e.
			\begin{equation} \label{integrineq}
				\int_M\left(\dfrac{224-9D}{3D}S^2
				+\dfrac{6D-128}{D}\lvert A^2\rvert^2
				+4cS-24c^2\right)dV_g= 0.
			\end{equation}
			Now, if we assume that $c=0$, we obtain
			\[
			\int_M\left(\dfrac{224-9D}{3D}S^2
			+\dfrac{6D-128}{D}\lvert A^2\rvert^2\right)dV_g= 0;
			\]
			if $D\leq\frac{64}{3}$, we can use the right-hand side of \eqref{ineq}
			to conclude that
			\begin{equation} \label{inequalitynod1}
				\left[\dfrac{224-9D}{3D}+
				\dfrac{7(3D-64)}{6D}\right]\int_MS^2dV_g=
				\dfrac{1}{2}\int_MS^2dV_g\leq 0,
			\end{equation}
			which guarantees that $M$ is totally geodesic, but this is a
			contradiction with the fact that $M$ is not locally conformally flat.
			If $D>\frac{64}{3}$, we exploit the left-hand side of \eqref{ineq}
			in order to obtain
			\[
			\left[\dfrac{224-9D}{3D}+
			\dfrac{3D-64}{2D}\right]\int_MS^2dV_g=
			\dfrac{256-9D}{6D}\int_MS^2dV_g\leq 0:
			\]
			hence, if $D<\frac{256}{9}$, we get that $M$ is totally geodesic. 
			Note that, \emph{a priori}, 
			we cannot conclude anything from the latest
			inequality if $D=\frac{256}{9}$: however, in this case,
			i.e. if 
			\[
			\int_M\lvert\operatorname{W}\rvert^2dV_g=\dfrac{256}{9}\pi^2\chi(M),
			\]
			we have 
			\[
			\dfrac{1}{8}\int_M\lvert\operatorname{W}\rvert^2dV_g=
			\int_M\left(-2\lvert\mathring{\operatorname{Ric}\rvert^2}+\dfrac{R^2}{6}\right)dV_g,
			\]
			that is,
			\[
			\int_M\left(-\dfrac{3}{8}S^2+
			\dfrac{3}{2}\lvert A^2\rvert^2\right)dV_g=0,
			\]
			which, by \eqref{riccisquare}, implies that $(M,g)$ is 
			an Einstein manifold and, hence, totally geodesic, which
			is again a contradiction (see e.g. \cite[Theorem 7.1]{Fialkow1938} or
			\cite[Theorem 3.1]{ryan}; we point out that the author in \cite{Fialkow1938} uses a different convention for the Ricci tensor). 
			
			Now, let $c=1$. Since $M$ is not 
			locally conformally flat, we have that 
			$S$ is constant and non-zero:
			then, \eqref{integrineq} becomes
			\[
			\int_M\left(\dfrac{224-9D}{3D}S^2
			+\dfrac{6D-128}{D}\lvert A^2\rvert^2
			+4S-24\right)dV_g= 0.
			\]
			If $D<\frac{64}{3}$, we can use again the
			right-hand side of \eqref{ineq} to obtain
			\begin{equation} \label{inequalitynod2}
				\int_M\left(\dfrac{1}{2}S^2+4S
				-24\right)dV_g\leq 0
			\end{equation}
			and, since $S$ is constant, 
			\[
			(S+12)(S-4)\leq 0.
			\]
			Thus, we have that $S\leq 4$: hence, \eqref{simonsid} implies that
			either $M$ is totally geodesic,
			which is impossible,
			or $S\equiv 4$ (\cite{simons}). 
			However, in the latter
			case, since equality holds in \eqref{ineq}, we would obtain
			that $(M,g)$ is locally conformally flat, which contradicts
			our assumption: therefore, we can conclude that $D\geq\frac{64}{3}$,
			which means that \eqref{boundsphere} holds.
			
			Last, if $D=\frac{64}{3}$ (i.e. equality holds in \eqref{boundsphere}), the
			coefficient of $\lvert A^2\rvert^2$ in \eqref{integrineq} vanishes and,
			therefore,
			we have that
			\[
			\int_M\left(\dfrac{1}{2}S^2
			+4S-24\right)dV_g=0
			\]
			without using \eqref{ineq}: hence, following the same line of reasoning
			as before, we can conclude that $S\in\{0,4\}$.
			However, $M$ cannot be locally conformally flat by assumption:
			therefore, by Corollary \ref{localisometrycor}, we conclude that 
			$M$ has to be locally isometric to the Clifford hypersurface
			\[
			\mathbb{S}^2\left(\dfrac{1}{\sqrt{2}}\right)\times\mathbb{S}^2\left(\dfrac{1}{\sqrt{2}}\right).
			\]
			We can observe that Corollary \ref{rigiditytopbound} follows immediately from
			the conclusions drawn in Theorem \ref{topbound}.
		\end{proof}
		\begin{rem}
			We point out that, if $c=1$, by \eqref{weylequal} and Corollary \ref{signaturehyper}, the optimal constant in \eqref{boundsphere} is the same appearing in \cite{GurskyAnnals}, while, if $c=0$, the constant in \eqref{boundeucl} is higher. In view of the result in \cite{GurskyAnnals}, the result for $c=1$ might suggest that, if $S$ is constant, then $R$ must be non-negative. 
			For the sake of the reader, we emphasize that, in \cite{GurskyAnnals}, the notation $\lvert\operatorname{W}\rvert^2$ represents the norm of the endomorphism $\mathcal{W}$ and, therefore, it is $1/4$ of 
			our definition of the squared norm of $\operatorname{W}$.
			
			We also note that, by Theorem \ref{equalweylnorms}, it is possible to obtain analogous bounds for $||\operatorname{W}^{\pm}||_{L^2}^2$. 
		\end{rem}
		\begin{rem}
			\begin{itemize}
				\item If $N^5(1)$ is the standard sphere $\mathbb{S}^5$ of radius $1$,
				we can conclude that, in the constant scalar curvature case, either $M$ is globally isometric to an equatorial sphere $\mathbb{S}^4$ or \eqref{boundsphere} holds, with equality if and only if 
				$M$ is isometric to a Clifford hypersurface \cite{lawson};
				\item we highlight the fact that, if $D\leq\frac{64}{3}$, 
				the dependance on $D$ disappears in \eqref{inequalitynod1} and \eqref{inequalitynod2}, suggesting the 
				``Simons-type behaviour" of the results in Theorem 
				\ref{topbound}. 
			\end{itemize}
		\end{rem}
		\noindent
		By Theorem \ref{topbound}, we can also show the validity of the improved bound
		stated in Corollary \ref{corbound}, in the case of non-positive scalar curvature,
		which mirrors the topological lower bound found by Gursky in \cite{gurskyzeit},
		under the hypothesis of the existence of a non-trivial conformal vector field.
		Again, we point out that in \cite{gurskyzeit} the author adopts the
		same convention as in \cite{GurskyAnnals} for $\lvert\operatorname{W}\rvert^2$.
		\begin{proof}[Proof of Corollary \ref{corbound}]
			First, we note that $(M,g)$ cannot be locally conformally flat: indeed,
			since $R\leq 0$, the totally geodesic space forms and the 
			Clifford hypersurface $\mathbb{S}^1\times\mathbb{S}^3$ cannot enter the picture. 
			Assume, as in the proof of Theorem \ref{topbound}, that $\chi(M)>0$ (otherwise,
			the claim is trivial): furthermore, suppose that 
			\[
			\int_M\lvert\operatorname{W}\rvert^2dV_g<32\pi^2\chi(M).
			\]
			Let $D>0$ such that 
			\[
			\int_M\lvert\operatorname{W}\rvert^2dV_g=D\pi^2\chi(M):
			\]
			by Corollary \ref{rigiditytopbound}, we know that 
			$D>64/3$, since $R\leq 0$. 
			Since $c=1$, by \eqref{ineq} and \eqref{integrineq}, we obtain that
			\[
			\int_M\dfrac{256-9D}{6D}S^2+4S -6\leq 0;
			\]
			following the same argument used to prove \eqref{boundsphere},
			we obtain
			\[
			S\leq\dfrac{12}{256-9D}\left(2\sqrt{2D(32-D)}-D\right),
			\quad \mbox{if} \quad \dfrac{64}{3}<D<\dfrac{256}{9} 
			\]
			and
			\[
			S\leq\dfrac{12}{9D-256}\left(D-2\sqrt{2D(32-D)}\right),
			\quad \mbox{if} \quad \dfrac{256}{9}<D<32.
			\]
			Since, in both cases, $S>4$, a simple computation
			allows us to conclude that
			\[
			4<S<12 \quad \mbox{if} \quad 
			\dfrac{64}{3}<D<\dfrac{256}{9} \vee \dfrac{256}{9}<D<32.
			\]
			Furthermore, if $D=\frac{256}{9}$, it is immediate to observe
			that
			\[
			4<S \leq 6:
			\]
			however, this is impossible, because $R\leq 0$ implies $S\geq 12$ by \eqref{scalconst},
			and this concludes the proof.
		\end{proof}
		
		The sharp bound for the $L^2$-norm of the Weyl tensor found in
		Theorem \ref{topbound} allows us to find an optimal lower bound
		for $S$ in the case of a closed minimal hypersurface with constant scalar curvature in $(N^5(1),g_N)$ (Corollary \ref{corpinch}). 
		\begin{proof}[Proof of Corollary \ref{corpinch}]
			By \eqref{boundsphere} and \eqref{ineq}, we have
			\begin{align*}
				\dfrac{4}{3}S^2\mathrm{Vol}_g(M)&=\dfrac{7}{3}S^2\mathrm{Vol}_g(M)
				-S^2\mathrm{Vol}_g(M)\geq \dfrac{7}{3}S^2\mathrm{Vol}_g(M)-
				4\int_M\lvert A^2\rvert^2dV_g=\\
				&=\int_M\left(\dfrac{7}{3}S^2-4\lvert A^2\rvert^2\right)dV_g=
				\int_M\lvert\operatorname{W}\rvert^2dV_g\geq \dfrac{64}{3}\pi^2\chi(M),
			\end{align*}
			which immediately implies \eqref{secondformbound}. Furthermore, by \eqref{ineq}, the equality case is trivial, since $(M^4,g)$ cannot be totally geodesic.
		\end{proof}
		
		
		The result provided by Corollary \ref{corpinch} is related to the so-called \emph{second pinching problem}: as we mentioned in the Introduction, it is conjectured that, given a minimal hypersurface in $\mathbb{S}^{n+1}$ with constant scalar curvature, the set of possible values for the scalar curvature is discrete: in particular, if $S>n$, the conjecture asserts that $S\geq 2n$. By the work of Peng and Terng \cite{pengterng}, it is known that, if $S>n$, there exists a dimensional constant $C(n)$ such that $S\geq n+C(n)$: to the best of our knowledge, the best constant so far, which is $C(n)=n/3$, was obtained in \cite{yangcheng}, while the existence of such constant in the non-constant scalar curvature case has been established in \cite{dingxin}, improving the pioneering work \cite{pengterngmathann}. 
		
		In this section we address this problem by finding lower bounds for $S$ in terms of the Euler characteristic of the minimal hypersurface in the four-dimensional case. 
		First, observe that, by \eqref{cherngaussbonnetminimal}, we easily obtain the following
		
		\begin{proposition} \label{quadrformula}
			Let $\mathcal{A}:=\mathrm{Vol}_g(M)^{-1}\int_M\lvert A^2\rvert^2 dV_g$; if $S$ is constant, then
			\begin{equation}
				S=\frac{2c}{3}+ \sqrt{-\frac{68}{9}c^2+\frac{32\pi^2\chi(M)}{3\mathrm{Vol}_g(M)}+2\mathcal{A}}. 
			\end{equation}
		\end{proposition}
		\begin{proof}
			The claim follows immediately by the quadratic formula
			\begin{align*}
				S&=\frac{4c\pm\sqrt{16c^2-12(24c^2-\frac{32\pi^2\chi(M)}{\mathrm{Vol}_g(M)}-6\mathcal{A})}}{6}\\
				&=\frac{2c}{3}\pm \sqrt{-\frac{68}{9}c^2+\frac{32\pi^2\chi(M)}{3\mathrm{Vol}_g(M)}+2\mathcal{A}}
			\end{align*}
		\end{proof}
		\begin{rem} \label{remquadrformula}
			It follows that
			\begin{equation}
				\frac{32\pi^2\chi(M)}{3\mathrm{Vol}_g(M)}+2\mathcal{A}\geq \frac{68}{9}c^2.
			\end{equation}
			Proposition \ref{quadrformula} implies that, in order to estimate $S,$ one may instead estimate $\frac{\chi(M)}{\mathrm{Vol}_g(M)}$ and $\mathcal{A}.$
		\end{rem}
			By using the Chern--Gauss--Bonnet formula \eqref{cherngaussbonnetminimal}, we can improve the estimate 
			proven in Corollary \ref{corpinch} by providing a lower bound for $S$ in terms of the Euler characteristic
			and the volume (Theorem \ref{sboundeuler} in the Introduction): we also prove that these 
			estimates are sharp. 
			\begin{proof}[Proof of Theorem \ref{sboundeuler}]
				Letting $\mathcal{A}$ be as in Proposition \ref{quadrformula} and
				$V:=\mathrm{Vol}_g(M)$, by the Chern--Gauss--Bonnet formula we have 
				\begin{equation} 
					\frac{3S^2}{4}-S+6-\frac{8\pi^2\chi(M)}{V}=\frac{3}{2}\mathcal{A}.
				\end{equation}
				By \eqref{ineq},
				it follows that
				\begin{equation}
					\frac{3S^2}{8}  \leq\frac{3S^2}{4}-S+6-\frac{8\pi^2\chi(M)}{V}\leq \frac{7S^2}{8}.  
				\end{equation}
				This gives us the following inequalities 
				\begin{align}
					0&\leq \frac{S^2}{8}+S+\frac{8\pi^2\chi(M)}{V}-6 \label{pineq}\\
					0&\leq \frac{3S^2}{8}-S-\frac{8\pi^2\chi(M)}{V}+6. \label{qineq}
				\end{align}
				Define
				\begin{align*}
					p_{k,v}(x)&=\frac{x^2}{8}+x+\frac{8\pi^2k}{v}-6,\\
					q_{k,v}(x)&=\frac{3x^2}{8}-x-\frac{8\pi^2k}{v}+6.
				\end{align*}
				It is clear that
				\begin{align*}
					p_{k,v}(x)&=0 \iff x=-4\pm8\sqrt{1-\frac{\pi^2k}{v}}\\
					q_{k,v}(x)&=0 \iff x=\frac{4}{3}\pm\frac{8\sqrt{2}}{3}\sqrt{\frac{3\pi^2k}{2v}-1}.
				\end{align*}
				\pagebreak
				
				\noindent
				It is worth to note that 
				\begin{itemize}
					\item the roots of $p_{k,v}$ are real if and only if 
					$k/v\leq\pi^2$
					\item the roots of $q_{k,v}$ are real if and only if
					$k/v\geq2/3\pi^2$,
				\end{itemize}
				which means that both polynomials have real roots if and only if 
				$2/3\pi^2\leq k/v\leq\pi^2$: we highlight the fact that,
				for every value of $k/v$, at least one of the
				polynomials $p_{k,v}$ and $q_{k,v}$ has real roots.
				
				In general we observe that, by \eqref{pineq} and \eqref{qineq}, $p_{\chi(M),V}(S)\geq 0$, 
				$q_{\chi(M),V}(S)\geq 0$ and $S\geq 4$ by Simons' identity
				\eqref{simonsid}, since $M^4$ is not totally geodesic. Furthermore, if 
				$k/v\geq2/3\pi^2$, the smallest root
				of $q_{k,v}$ cannot be greater or equal than $4$: hence, since 
				$p_{\chi(M),V}$ has only one positive real root if $k/v\leq\pi^2$, by \eqref{secondformbound}, \eqref{pineq} and
				\eqref{qineq} we have
				\begin{equation} \label{smax}
					S\geq \max\left(\left\{4\pi\sqrt{\dfrac{\chi(M)}{V}},-4+8\sqrt{1-\dfrac{\pi^2\chi(M)}{V}},
					\dfrac{4}{3}+\dfrac{8\sqrt{2}}{3}\sqrt{\dfrac{3\pi^2\chi(M)}{2V}-1}\right\}\cap\mathbb{R}\right).
				\end{equation}
				First, observe that, if $\chi(M)/V>1/\pi^2$, the second term
				in \eqref{smax} is not a real number: moreover, it can
				be easily checked that, in this interval,
				\[
				\dfrac{4}{3}+\dfrac{8\sqrt{2}}{3}\sqrt{\dfrac{3\pi^2\chi(M)}{2V}-1}
				> 4\pi\sqrt{\dfrac{\chi(M)}{V}}.
				\]
				Now, let $\chi(M)/V\leq 1/\pi^2$: we note that, 
				if $\chi(M)/V<2/3\pi^2$, the third term in \eqref{smax} is not
				real and a simple computation proves that
				\[
				4\pi\sqrt{\dfrac{\chi(M)}{V}}\geq -4+8\sqrt{1-\dfrac{\pi^2\chi(M)}{V}}
				\quad \mbox{ if and only if} \quad 
				\dfrac{9}{25\pi^2}\leq \dfrac{\chi(M)}{V}<\dfrac{2}{3\pi^2}.
				\]
				Finally, if $2/3\pi^2\leq \chi(M)/V\leq 1/\pi^2$, all three
				terms in \eqref{smax} are well-defined: in this case, we have
				\[
				4\pi\sqrt{\dfrac{\chi(M)}{V}}\geq
				\dfrac{4}{3}+\dfrac{8\sqrt{2}}{3}\sqrt{\dfrac{3\pi^2\chi(M)}{2V}-1}>-4+8\sqrt{1-\dfrac{\pi^2\chi(M)}{V}}.
				\]
				The equality cases follow trivially by Corollary \ref{corpinch}, \eqref{pineq} and \eqref{qineq}.
			\end{proof}
		\begin{rem}
			We just point out that, if we assume $\pi^2\chi(M)\geq V$, we can obtain another
			lower bound for the $L^2$-norm of the Weyl tensor: indeed, since
			$S$ is constant, by means of \eqref{cherngaussbonnetminimal} and 
			\eqref{ineq}, we
			obtain
			\[
			32\pi^2\chi(M)\leq \int_M\lvert\operatorname{W}\rvert^2dV_g
			+
			\left(\dfrac{S^2}{6}-4S+24\right)V=\int_M
			\lvert\operatorname{W}\rvert^2dV_g+\dfrac{1}{6}(S-12)^2V,
			\]
			which, by hypothesis, yields
			\[
			\int_M\lvert\operatorname{W}\rvert^2 dV_g\geq\pi^2\left[
			32-\dfrac{1}{6}(S-12)^2\right]\chi(M).
			\]
			We can compare the right-hand side to the one appearing in
			\eqref{boundsphere} and observe that 
			\[
			32-\dfrac{1}{6}(S-12)^2 > \dfrac{64}{3}, \quad 
			\mbox{ if } 4\leq S\leq 20.
			\]
			Furthermore, by \eqref{scalconst} the maximum value of the right-hand side in terms of 
			$S$ is obtained when $(M^4,g)$ is scalar-flat: in this case, however, we obtain the same topological lower bound proven in Corollary \ref{corbound}.
		\end{rem}
		Theorem \ref{sboundeuler} suggests that, if one manages to 
		obtain useful bounds for the volume of $M$, we can obtain lower bounds for $S$ in terms of $\chi(M)$ that might improve the 
		best constants present in the literature \cite{yangcheng}, which is the core of Corollary \ref{corsboundeuler}. 
		In order to prove it, we have to give some preliminary definitions. Let 
		$\mathbb{S}^5\subset\mathbb{R}^6$ be the standard sphere of radius $1$ centered at the origin: 
		for every $a\in\mathbb{S}^5$, we can define $\phi_a(x)=<x,a>$, where
		$x\in\mathbb{R}^6$. The set 
		\[
		Z_a:=\{x\in\mathbb{R}^6:\phi_a(x)=0\}
		\]
		defines a hyperplane in $\mathbb{R}^6$ which passes through the
		origin and, therefore, intersects $\mathbb{S}^5$ exactly in an equatorial $4$-sphere, which we denote by $\mathbb{S}_a^4$.
		If we restrict $\phi_a$ to a closed, minimal hypersurface $M^4$
		of $\mathbb{S}^5$, we obtain that $Z_a$ restricts
		to $Z_a^M:=M^4\cap\mathbb{S}_a^4$.
		In this context, we can use Theorem \ref{sboundeuler} and impose a bound on the volumes of $Z_a^M$ in order to improve the known lower bounds on $S$, under some topological conditions.
		\begin{proof}[Proof of Corollary \ref{corsboundeuler}]
			Let again $V:=\mathrm{Vol}_g(M)$: by \cite[Theorem 1.1]{chenli}, we know that
			\[
			V\leq\dfrac{5\lvert\mathbb{S}^5\rvert}{4\lvert\mathbb{S}^4\rvert}\sup_{a\in\mathbb{S}^5}\mathrm{Vol}_g(Z_a^M),
			\]
			which, by hypothesis, yields
			\begin{equation} \label{unifvolbound}
				V\leq\dfrac{5\lvert\mathbb{S}^5\rvert}{4}=\dfrac{5\pi^3}{4}.
			\end{equation}
			Now, if $\chi(M)\leq 0$, by Theorem \ref{sboundeuler} and 
			\eqref{unifvolbound} we get
			\[
			S\geq -4+8\sqrt{1-\dfrac{\pi^2\chi(M)}{V}}\geq 
			-4+8\sqrt{1-\dfrac{4\chi(M)}{5\pi}},
			\]
			which, if $\chi(M)\leq -2$, implies
			\[
			S\geq -4+8\sqrt{1+\dfrac{8}{5\pi}}>\dfrac{16}{3}.
			\]
			Now, assume that $\chi(M)\geq 0$: by hypothesis, 
			\[
			\dfrac{\chi(M)}{V}\geq \dfrac{4\chi(M)}{5\pi},
			\]
			which means that, if $\chi(M)\geq 4$,
			\[
			\dfrac{\chi(M)}{V}\geq \dfrac{16}{5\pi^3}>\dfrac{1}{\pi^2}
			\]
			and, by Theorem \ref{sboundeuler},
			\[
			S > \dfrac{4}{3}+\dfrac{8\sqrt{2}}{3}\sqrt{\dfrac{3\pi^2\chi(M)}{2V}-1}\geq \dfrac{4}{3}+\dfrac{8\sqrt{2}}{3}\sqrt{\dfrac{6\chi(M)}{5\pi}-1}.
			\]
			Note that right-hand side is well-defined, since $\chi(M)\geq 4$:
			furthermore, if $\chi(M)>4$ (i.e. $\chi(M)\geq 6$), we have that
			\[
			S > \dfrac{4}{3}+\dfrac{8\sqrt{2}}{3}\sqrt{\dfrac{36-5\pi}{5\pi}}>\dfrac{16}{3}.
			\]
		\end{proof}
		\begin{rem}
			\begin{itemize}
				\item If we assume that $\chi(M)=2$, the argument of
				Corollary \ref{corsboundeuler} does not work and we just obtain $S\geq 4$: indeed, by using \eqref{unifvolbound} we would obtain
				\[
				\dfrac{2}{V}\geq \dfrac{8}{5\pi^3}>\dfrac{9}{25\pi^2}.
				\]
				By Theorem \ref{sboundeuler}, if $2\pi^2\leq V<(50/9)\pi^2$,
				we have that
				\[
				S\geq4\sqrt{\dfrac{8}{5\pi}},
				\]
				which does not tell us anything, since the right-hand side is less than $4$. 
				On the other hand, if $V<2\pi^2$, we would have
				\[
				S\geq \dfrac{4}{3}+\dfrac{8\sqrt{2}}{3}\sqrt{\dfrac{3\pi^2}{V}-1}>4,
				\]
				which, again, is already known;
				\item let us assume that the hypothesis on $\mathrm{Vol}_g(Z_a^M)$ in Corollary \ref{corsboundeuler} does not hold. If $\chi(M)\geq 0$ and $\chi(M)/V\leq 9/25\pi^2$, we can still obtain a lower bound for $S$: indeed, in \cite[Corollary 5]{chengliyau} Cheng, Li and Yau proved that, in our case, 
				\[
				V\geq \left(1+\dfrac{1}{B_4}\right)\lvert\mathbb{S}^4\rvert,
				\]
				where $B_4=11+2e^{128}$. Using this estimate, by 
				Theorem \ref{sboundeuler} we obtain
				\[
				S\geq -4+8\sqrt{1-\dfrac{\pi^2\chi(M)}{V}}\geq 
				-4+8\sqrt{1-\dfrac{3\chi(M)}{8\left(1+\frac{1}{11+2e^{128}}\right)}},
				\]
				where the right-hand side is only defined when 
				$\chi(M)\in\{0,2\}$. However, in this case we would simply get
				$S\geq 4$, which is not a satisfactory lower bound. 
			\end{itemize}
		\end{rem}
		
		\section{Bochner-type formulas and rigidity results} \label{bochnersection}
		
		In this final section, by means of some Bochner--Weitzenb\"{o}ck formulas, we
		derive some rigidity results for minimal hypersurfaces in the standard sphere, under some curvature hypotheses concerning the Weyl tensor. 
		Given a $(r,s)$-tensor field $\operatorname{T}$, we denote the components of its covariant derivative $\nabla\operatorname{T}$ with respect to a local orthonormal coframe $\{\theta^i\}$ as $T_{i_1...i_s,k}^{j_1...j_r}$, where
		\[
		T_{i_1...i_s,k}^{j_1...j_r}\theta^k=dT_{i_1...i_s}^{j_1...j_r}-\sum_{l=1}^sT_{i_1...i_{l-1}ti_{l+1}...i_{s}}^{j_1...j_r}\theta_{i_l}^t+\sum_{l=1}^rT_{i_1...i_{s}}^{j_1...j_{l-1}tj_{l+1}...j_r}\theta_t^{j_l},
		\]
		where $T_{i_1...i_s}^{j_1...j_r}$ are the local components of $\operatorname{T}$ and $\theta_j^i$ are the \emph{Levi-Civita connection forms} of the chosen coframe (see e.g. \cite{CatinoMastroliaBook}).
		The only exception will be made for the second fundamental form $A$: in this case, we will denote the local components of $\nabla A$ as $A_{ijk}$ instead of $A_{ij,k}$, in accordance with the literature.
		
		First, we want to prove sharp integral bounds for the squared norm of $\nabla A^2$ (Theorem \ref{sharpquadrbound} and 
		Corollary \ref{sharpquadrcor}), in order to provide an integral characterization of totally geodesic spheres and Clifford hypersurfaces in $\mathbb{S}^{n+1}$ (or, more precisely, in a positively curved space form $(N^{n+1}(1),g_N)$). In order to do so, we show the validity of useful Bochner--Weitzenb\"{o}ck formulas, which can be seen as extensions of Simons' identity \eqref{simonsid}. 
		\begin{proof}[Proof of Theorem \ref{sharpquadrbound} and Corollary \ref{sharpquadrcor}]
			First, recall the standard formula
			\[
			\dfrac{1}{2}\Delta f^2=\lvert\nabla f\rvert^2+f\Delta f,
			\]
			holding for every $f\in C^2(M)$. Then, we have
			\[
			\dfrac{1}{2}\Delta S^2=\lvert\nabla S\rvert^2+S\Delta S,
			\]
			which, by \eqref{simonsid}, becomes
			\begin{equation} \label{simons2}
				\dfrac{1}{2}\Delta S^2=2S\lvert\nabla A\rvert^2+2S^2(nc-S)+\lvert\nabla S\rvert^2.
			\end{equation}
			Now, by \eqref{simonsidlocal}, we can compute
			\begin{align*}
				A_{ij,k}^2&=(A_{it}A_{tj})_k=A_{itk}A_{tj}+A_{it}A_{tjk}\\
				&\Longrightarrow \Delta A_{ij}^2=A_{jt}\Delta A_{it}+A_{it}\Delta A_{jt}+2A_{ikt}A_{jkt}=
				2(nc-S)A_{ij}^2+2A_{ikt}A_{jkt};
			\end{align*}
			recalling that
			\[
			\dfrac{1}{2}\Delta\lvert A^2\rvert^2=\lvert\nabla A^2\rvert^2+A_{ij}^2\Delta A_{ij}^2,
			\]
			we obtain
			\begin{equation} \label{simonsasquare}
				\dfrac{1}{2}\Delta\lvert A^2\rvert^2=\lvert\nabla A^2\rvert^2+2\lvert A^2\rvert^2(nc-S)+2A_{ij}^2A_{ikt}A_{jkt}.
			\end{equation}
			Now, let $c=1$: we integrate by parts the last term of the right-hand side of \eqref{simonsasquare} and use \eqref{simonsidlocal} to obtain
			\begin{align*}
				2\int_MA_{ij}^2A_{ikt}A_{jkt}dV_g&=-2\int_M\left(A_{ij,k}^2A_{it}A_{jtk}+A_{ij}^2A_{it}\Delta A_{jt}\right)dV_g=\\
				&=-2\int_M\left[
				(A_{isk}A_{sj}+A_{is}A_{sjk})A_{it}A_{jtk}+\lvert A^2\rvert^2(n-S)\right]dV_g,
			\end{align*}
			which, by renaming the indices in the second term inside the round brackets, implies that
			\[
			2\int_MA_{ij}^2A_{ikt}A_{jkt}dV_g=-
			\int_MA_{it}A_{sj}A_{isk}A_{jtk}dV_g+\int_M\lvert A^2\rvert^2(S-n)dV_g.
			\]
			\noindent
			Let us define the tensor $\operatorname{T}$ locally
			expressed by
			\[
			T_{ijk}:=A_{is}A_{sjk};
			\]
			we observe that, since $A$ is a Codazzi tensor, $\operatorname{T}$ is symmetric
			with respect to the last two indices.
			Applying the definition of $\operatorname{T}$ and the Cauchy-Schwarz inequality we obtain
			\[
			-\int_MA_{it}A_{sj}A_{isk}A_{jtk}\,dV_g=
			-\int_MT_{tsk}T_{stk} \,dV_g\geq
			-\int_M\lvert T\rvert^2dV_g =
			- \int_MA_{ij}^2A_{ikt}A_{jkt}\,dV_g,
			\]
			i.e.
			\begin{equation} \label{intestimtrace}
				3\int_M A_{ij}^2A_{ikt}A_{jkt}dV_g
				\geq \int_M\lvert A^2\rvert^2(S-n)dV_g:
			\end{equation}
			integrating \eqref{simonsasquare} and using \eqref{intestimtrace} and
			the hypothesis,
			we obtain the right-hand side of \eqref{boundasquare}.
			
			Now, we observe that $A_{ij}^2$ and $K_{ij}:=A_{ikt}A_{jkt}$ are symmetric and positive semidefinite as matrices: hence, as a consequence of Cauchy-Schwarz inequality and \eqref{weylminimalgeneral2}, we get
			\[
			A_{ij}^2A_{ikt}A_{jkt}=\lvert A_{ij}^2A_{ikt}A_{jkt}\rvert\leq\lvert A^2\rvert\cdot\lvert\operatorname{K}\rvert 
			\leq\sqrt{\dfrac{n^2-3n+3}{n(n-1)}} S\cdot\lvert\nabla A\rvert^2. 
			\]
			Here we used the fact that, given a symmetric, positive semidefinite matrix $B$ in an $n$-dimensional real 
			inner product space $(V,\langle,\rangle)$, 
			if $\mu_i$ are the eigenvalues of $B$, then
			\[
			|B|=\sqrt{\sum_{i=1}^n\mu_i^2}\leq\sum_{i=1}^n\mu_i=\operatorname{tr}(B).
			\]
			By integrating \eqref{simonsasquare}, we obtain
			\[
			0\leq\int_M\left(\lvert\nabla A^2\rvert^2-2\lvert A^2\rvert^2(S-n)+2\sqrt{\dfrac{n^2-3n+3}{n(n-1)}}S\cdot\lvert\nabla A\rvert^2\right)dV_g;
			\]
			now, we use \eqref{simons2} to get
			\begin{equation*}
				0\leq\int_M\lvert\nabla A^2\rvert^2-2\lvert A^2\rvert^2(S-n)+\sqrt{\dfrac{n^2-3n+3}{n(n-1)}}[2S^2(S-n)-\lvert\nabla S\rvert^2]dV_g
			\end{equation*}
			and, hence, the left-hand side of \eqref{boundasquare}. 
			
			Now, we study the equality cases.
			\begin{enumerate}
				\item By our previous argument, equality holds in the right-hand side of \eqref{boundasquare} if and only if 
				\[
				T_{jik}T_{ijk}=\lvert\operatorname{T}\rvert^2=T_{ijk}T_{ijk}
				\]
				on $M$. 
				It is straightforward to observe that,
				in this case, $\operatorname{T}$ is
				also symmetric with respect to the
				first two indices, i.e.
				\[
				T_{ijk}=T_{jik}=T_{ikj}.
				\]
				Using the symmetries of $\operatorname{T}$, we can compute
				\begin{align*}
					A_{ij,k}^2-A_{ik,j}^2&=
					(A_{is}A_{sj})_k-(A_{is}A_{sk})_j=
					A_{isk}A_{sj}+A_{is}A_{sjk}-
					A_{isj}A_{sk}-A_{is}A_{skj}=\\
					&=T_{jik}+T_{ijk}-T_{kij}-T_{ikj}=0,
				\end{align*}
				which means that $A^2$ is a Codazzi tensor: 
				by a result due to Umehara \cite[Theorem 1.3]{umehara} we can conclude that 
				$(M^n,g)$ must also be isoparametric, with $\nabla A\equiv 0$: by Proposition \ref{localisometryprop}, the claim follows.
				\item If $(M,g)$ is either totally geodesic or locally isometric to a Clifford hypersurface, the equality
				in the left-hand side of \eqref{boundasquare} can be immediately obtained by means of Proposition \ref{localisometryprop}. 
				Conversely, assume that equality in \eqref{boundasquare} holds. By the equality in \eqref{weylminimalgeneral2},
				we immediately obtain that $(M^4,g)$ is locally conformally flat.
				Assume that there exists $p\in M$ such that
				$A_{ij}^2K_{ij}$ is not zero. By Cauchy-Schwarz, there exists $\lambda>0$ such that
				\begin{equation} \label{proportional}
					K_{ij}=\lambda\cdot A_{ij}^2,
				\end{equation}
				for every $i,j$. 
				Tracing \eqref{proportional}, we obtain
				\[
				\lvert\nabla A\rvert^2=\lambda\cdot S:
				\]
				since $\lambda>0$, contracting \eqref{proportional} with $A_{ij}^2$,
				we have
				\[
				\lambda\cdot\lvert A^2\rvert^2= A_{ij}^2K_{ij}= S\cdot \lvert\nabla A\rvert^2=\lambda\cdot S^2,
				\]
				which, by local conformal flatness, implies
				\[
				\dfrac{n^2-3n+3}{n(n-1)}S^2=\lvert A^2\rvert^2=S^2,
				\]
				contradicting the fact that $A_{ij}^2K_{ij}\neq0$ at $p$.
				
				Therefore, $A_{ij}^2K_{ij}=S\cdot\lvert\nabla A\rvert^2=0$ on $M$. 
				If $\nabla A\equiv 0$ on $M$, the conclusion follows directly
				from Proposition \ref{localisometryprop}, hence let us assume that there exists $q\in M$ such that 
				$\nabla A\not\equiv 0$ at $q$: by continuity, there exists an open neighborhood $U$ of $q$ in which 
				$\nabla A\not\equiv 0$, which implies that $S\equiv 0$ in $U$. On the other hand, if 
				$S\neq 0$ at some point $q'$, there exists an open neighborhood $U'$ of $q'$ in which $\nabla A\equiv 0$, which means
				that $S$ is constant on $U'$. By continuity, the only possibility is that $S\equiv 0$ on $M$, which contradicts
				the existence of a point $q$ at which $\nabla A\not\equiv 0$, and this
				concludes the proof of the Theorem. 
				
				In order to prove Corollary \ref{sharpquadrcor}, we observe that the inequality \eqref{refinedboundasquare} is immediately deduced from \eqref{boundasquare}, using \eqref{weylminimalgeneral2} and \eqref{tracriccmingen}. The claim on the equality case follows from the characterizations of the equalities in Theorem \ref{sharpquadrbound}.
			\end{enumerate}
		\end{proof}
		\begin{rem}
			By Simons' identity \eqref{simonsidint}, we know that, if $(M^4,g)$ is a closed hypersurface of a positive space form $(N^5(1),g_N)$ and 
			$S\leq n$, then either $S\equiv 0$ or $S\equiv n$ and $(M^4,g)$ is isoparametric. We observe that it is possible to 
			prove the same result using \eqref{simonsasquare}. Indeed, we know that $A_{ij}^2A_{ikt}A_{jkt}$ is non-negative on $M$,
			since it can be seen as the trace of the product of two positive semidefinite matrices: hence, integrating \eqref{simonsasquare}
			and using the fact that $\lvert A^2\rvert^2(n-S)$ and $A_{ij}^2A_{ikt}A_{jkt}$ are non-negative, we immediately get that
			$\nabla A^2\equiv 0$: exploiting again \eqref{simonsasquare}, we also obtain that 
			$\lvert A^2\rvert^2(n-S)=A_{ij}^2A_{ikt}A_{jkt}=0$, which proves the claim. 
		\end{rem}
		For the last results of the section, we will need a useful Proposition proven in \cite[Lemma 2.4]{huisken} and, later, in \cite[Proposition 2.2]{case2023sharp}, which can be seen as a generalization of the sharp bounds described in \cite{pengterng}:
		
		\begin{proposition}
			Let $V$ be a real $n$-dimensional inner product space, $n\geq 3$, and $A\in\operatorname{End}(V)$ be a symmetric, trace-free endomorphism of $V$. Then, we have that
			\[
			-\dfrac{(n-2)}{\sqrt{n(n-1)}}\lvert A\rvert^3\leq \operatorname{tr}(A^3)\leq \dfrac{(n-2)}{\sqrt{n(n-1)}}\lvert A\rvert^3;
			\]
			moreover, one of the equalities hold if and only if $A$ has an eigenspace of dimension at least $n-1$.
		\end{proposition}
		\noindent
		As a consequence, in our setting, we have the following
		\begin{cor} \label{lagrangecor}
			Let $(M^4,g)$ be a minimally immersed hypersurface in a space form $(N^5(c),g_N)$. Then, at every $p\in M$,
			\begin{equation} \label{lagrange}
				-\sqrt{\dfrac{S^3}{3}}\leq\operatorname{tr(A^3)}\leq\sqrt{\dfrac{S^3}{3}}
			\end{equation}
			and equality holds at $p$ if and only if $A$ has a principal curvature of multiplicity at least $3$ at $p$.
		\end{cor}
		

		Now, let $(M,g)$ be a Riemannian manifold of dimension $n\geq 4$. We say that
		$g$ is \emph{harmonic Weyl} if the Weyl tensor $\operatorname{W}$ is
		divergence free, i.e. if
		\[
		\delta\operatorname{W}\equiv 0, \mbox{ on } M.
		\]
		In components, this can be rewritten as
		\[
		(\delta\operatorname{W})_{ijk}=W_{tijk,t}=0;
		\]
		since $n\geq 4$, this condition is equivalent to the vanishing of the so-called \emph{Cotton tensor} $\operatorname{C}$, locally expressed
		by
		\begin{equation} \label{cottondef}
			C_{ijk}=R_{ij,k}-R_{ik,j}-\dfrac{1}{2(n-1)}(R_k\delta_{ij}-R_j\delta_{ik})=
			P_{ij,k}-P_{ik,j},
		\end{equation}
		where $P_{ij}$ are the local components
		of the \emph{Schouten tensor}
		$\operatorname{P}$.
		Indeed, since 
		$n\geq 4$, we can express 
		$\operatorname{C}$ in terms of 
		the Weyl tensor $\operatorname{W}$ as
		(see e.g. \cite{CatinoMastroliaBook})
		\begin{equation} \label{cottonweyl}
			C_{ijk}=-\dfrac{n-2}{n-3}(\delta\operatorname{W})_{ijk}.
		\end{equation}
		On one hand, by \eqref{cottonweyl} the harmonic Weyl condition is trivially satisfied
		on any locally conformally flat manifold; on the other hand, by \eqref{cottondef} the same
		condition also holds on every Einstein manifold, since they all are also
		Ricci parallel (i.e. $\nabla\operatorname{Ric}\equiv 0$). 
		
		If $n=4$, by \eqref{weyldecomp}, we can introduce the notion of
		\emph{half harmonic Weyl} metric, i.e. such that $\operatorname{W}^+$
		or $\operatorname{W}^-$ are divergence free. Recall that the action of 
		the Hodge star operator commutes with the covariant differentiation of
		$2$-forms, hence
		the decomposition of $\Lambda^2=\Lambda_+\oplus\Lambda_-$ is preserved by the
		Levi-Civita connection, which implies that
		\begin{equation*}
			\nabla\operatorname{W}=\nabla\operatorname{W}^++\nabla\operatorname{W}^-,
		\end{equation*}
		and, therefore,
		\begin{equation*}
			\delta(\operatorname{W})=\delta(\operatorname{W})^++\delta(\operatorname{W})^-.
		\end{equation*}
		With respect to a local orthonormal (co)frame, we can 
		express the divergence of $\operatorname{W}^{\pm}$ as follows (see e.g. \cite{CatinoDamenoMastrolia} for an explicit expression of these components):
		\[
		(\delta\operatorname{W}^{\pm})_{ijk}=W_{tijk,t}^{\pm}.
		\]
		Using again the fact that the Hodge star operator commutes with the covariant
		derivative induced by the Levi-Civita connection,
		by \eqref{hodgeweylplusmincomp} we obtain
		\begin{equation} \label{divergplusmin}
			(\delta\operatorname{W}^{\pm})_{ijk}=
			\dfrac{1}{2}[(\delta\operatorname{W})_{ijk}\pm
			\star(\delta(\operatorname{W})_{ijk})]=
			\dfrac{1}{2}\left[W_{tijk,t}\pm\dfrac{1}{2}\mu_{jkrs}W_{tirs,t}\right].
		\end{equation}

		It is well known that, if $(M,g)$ is a half harmonic Weyl manifold, the
		(anti-)self-dual Weyl tensor satisfies a remarkable Bochner--Weitzenb\"{o}ck formula:
		in dimension four, it can be locally expressed as follows 
		(\cite{Derdzinski}; see also \cite{CatinoMastroliaBook}):
		\begin{equation} \label{bochnercomp}
			\Delta W_{ijkl}^{\pm}=\dfrac{R}{2}W_{ijkl}^{\pm}-2(W_{ipjq}^{\pm}W_{pqkl}^{\pm}-
			W_{iplq}^{\pm}W_{jpkq}^{\pm}+W_{ipkq}^{\pm}W_{jplq}^{\pm}).
		\end{equation}
		Using the first Bianchi identity, one can show that (\cite{CatinoMastroliaBook})
		\begin{equation} \label{eqweyl1}
			W_{ipjq}^{\pm}W_{pqkl}^{\pm}=\dfrac{1}{2}W_{ijpq}^{\pm}W_{pqkl}^{\pm};
		\end{equation}
		moreover, exploiting \eqref{symbolcontraction} and \eqref{hodgeweylplusmincomp}, it is possible to prove that
		\begin{equation} \label{eqweyl2}
			W_{ipkq}^{\pm}W_{jplq}^{\pm}+W_{iplq}^{\pm}W_{jpkq}^{\pm}=\dfrac{1}{8}\lvert\operatorname{W}^{\pm}\rvert^2\delta_{ij}\delta_{kl}.
		\end{equation}
		Combining \eqref{eqweyl1} and \eqref{eqweyl2} and observing
		that $i\neq j$ in \eqref{bochnercomp}, we can rewrite it as follows:
		\begin{equation} \label{bochnercompshort}
			\Delta W_{ijkl}^{\pm}=\dfrac{R}{2}W_{ijkl}^{\pm}-W_{ijpq}^{\pm}W_{pqkl}^{\pm}-2W_{ipkq}^{\pm}W_{jplq}^{\pm}.
		\end{equation}
		Now, contracting \eqref{bochnercompshort} with $W_{ijkl}$ and recalling
		that, for every tensor field $\operatorname{T}$,
		\[
		\dfrac{1}{2}\Delta\lvert\operatorname{T}\rvert^2=\lvert\nabla\operatorname{T}\rvert^2+
		\langle\operatorname{T},\Delta\operatorname{T}\rangle,
		\]
		we obtain 
		\begin{equation} \label{bochscalar}
			\dfrac{1}{2}\Delta\lvert\operatorname{W^{\pm}}\rvert^2=\lvert\nabla\operatorname{W^{\pm}}\rvert^2+\dfrac{R}{2}\lvert\operatorname{W^{\pm}}\rvert^2-3W_{ijkl}^{\pm}W_{pqkl}^{\pm}W_{ijpq}^{\pm}.
		\end{equation}
		We observe that, by the orthogonality property of \eqref{weyldecomp},
		both \eqref{bochnercompshort} and \eqref{bochscalar} hold 
		also for $\operatorname{W}$ in the harmonic Weyl case (for
		useful applications of these formulas, see, for instance, \cite{Derdzinski, GurskyAnnals, Gurskyharmonicweyl, gurskylebrun, PengWu}). 
		
		In the special context of minimal hypersurfaces immersed into space forms, we are able to prove an integral inequality for
		half harmonic Weyl induced metrics (Theorem \ref{harmweylinequality}).
		
		\begin{proof}[Proof of Theorem \ref{harmweylinequality}]

			Without loss of generality, we assume that $\delta\operatorname{W}^+\equiv 0$ and that $(M^4,g)$ is not locally conformally flat. 
			By integrating both sides of \eqref{bochscalar} for $\operatorname{W}^+$ we get 
			\begin{equation} \label{integrweitz}
				\int_M |\nabla\operatorname{W}^+|^2 dV_g= 3\int_M W_{ijkl}^+W_{pqkl}^+W_{ijpq}^+ dV_g-\int_M\frac{12c-S}{2}|\operatorname{W}^+|^2 dV_g,
			\end{equation}
			which,
			by \eqref{weylnormhyper}, we rewrite as 
			\begin{equation}
				\int_M |\nabla \operatorname{W}^+|^2 dV_g= 3\int_M W_{ijkl}^+W_{klpq}^+W_{ijpq}^+ dV_g-\int_M\frac{12c-S}{2}\bigg( \frac{7}{6}S^2-2|A^2|^2\bigg) dV_g.
			\end{equation}
			In order to simplify the computations, we will deal with the components of the full Weyl tensor $\operatorname{W}$. We use \eqref{weylfialkow} to express the cubic term by means of the second fundamental form: first,
			we compute
			\begin{align*}
				W_{ijkl}W_{pqkl}&=\left(A_{ik}A_{jl}-A_{il}A_{jk}+
				F_{ik}\delta_{jl}-F_{il}\delta_{jk}+F_{jl}\delta_{ik}-F_{jk}\delta_{il}\right)\\
				&\ast \left(A_{pk}A_{ql}-A_{pl}A_{qk}+
				F_{pk}\delta_{ql}-F_{pl}\delta_{qk}+F_{ql}\delta_{pk}-F_{qk}\delta_{pl}\right)=\\
				&=2\left[A_{ip}^2A_{jq}^2-
				A_{iq}^2A_{jp}^2+A_{ik}A_{jq}F_{pk}-
				A_{ik}A_{jp}F_{qk}+A_{ip}A_{jk}F_{qk}-A_{iq}A_{jk}F_{pk}\right.+\\
				&+A_{pk}A_{jq}F_{ik}-A_{qk}A_{jp}F_{ik}+
				A_{ip}A_{qk}F_{jk}-A_{iq}A_{pk}F_{jk}+F_{ip}^2\delta_{jq}-
				F_{iq}^2\delta_{jp}\left.\right.+\\
				&+\left.F_{jq}^2\delta_{ip}-F_{jp}^2\delta_{iq}
				+2F_{ip}F_{jq}-2F_{iq}F_{jp}\right].
			\end{align*}
			Now, taking into account the symmetries of $\operatorname{W}$, $A$ and $F$, recalling that $\operatorname{W}$ is a traceless tensor and renaming the indices of the different terms in the previous equation, we can realize that
			\[
			W_{ijkl}W_{pqkl}W_{ijpq}=(4A_{ip}^2A_{jq}^2+16A_{ik}A_{jq}F_{pk}+8F_{ip}F_{jq})W_{ijpq}:
			\]
			observing that, by \eqref{fialkow},
			\[
			\operatorname{tr}(\operatorname{F})=\dfrac{1}{6}S\,\quad \mbox{ and } \quad
			\lvert\operatorname{F}\rvert^2=\dfrac{1}{4}\left(\lvert A^2\rvert^2-\dfrac{2}{9}S^2\right),
			\]
			we can finally compute
			\begin{align*}
				W_{ijpq}W_{ijkl}W_{pqkl}&=(4A_{ip}^2A_{jq}^2+16A_{ik}A_{jq}F_{pk}+8F_{ip}F_{jq})\ast\\
				&\ast\left(A_{ip}A_{jq}-A_{iq}A_{jp}+
				F_{ip}\delta_{jq}-F_{iq}\delta_{jp}+F_{jq}\delta_{ip}-F_{jp}\delta_{iq}\right)=\\
				&=10[\operatorname{tr}(A^3)]^2
				-28\operatorname{tr}(A^6)+19S\lvert A^2\rvert^2
				-\dfrac{23}{9}S^3.
			\end{align*}
			Around a point $p\in M$ we can choose a Darboux frame such that $\mathcal{W}^+$ is in diagonal form at $p$: in this case, 
			by \eqref{weyleigen} the eigenvalues of $\mathcal{W}^+$ and $\mathcal{W}^-$ are equal, thus, if $i\neq j\neq k\neq l$
			we have that
			\[
			W_{ijij}=2W_{ijij}^+, \quad W_{ijik}=W_{ijkl}=0.
			\]
			Thus, by \eqref{hodgeweylplusmin}, we can obtain
			\begin{align*}
				W_{ijkl}W_{pqkl}W_{ijpq}&=16(W_{1212})^3+16(W_{1313})^2+16(W_{1414})^2=\\
				&=128\left[
				(W_{1212}^+)^3+(W_{1313}^+)^3+(W_{1414}^+)^3\right]=\\
				&=2\left[16(W_{1212}^+)^3+48(W_{1212}^+)(W_{1234}^+)^2+16(W_{1313}^+)^3+48(W_{1313}^+)(W_{1342}^+)^2\right.+\\ &\left.+16(W_{1414}^+)^3+48(W_{1414}^+)(W_{1423}^+)^2\right]=\\
				&=2W_{ijkl}^+W_{pqkl}^+W_{ijpq}^+;
			\end{align*}
			since $W_{ijkl}^+W_{pqkl}^+W_{ijpq}^+$ is a smooth function on $M$, its values do not depend on the choice of coordinate system, hence we get
			\[
			W_{ijkl}^+W_{pqkl}^+W_{ijpq}^+=5[\operatorname{tr}(A^3)]^2-14\operatorname{tr}(A^6)+\dfrac{19}{2}S\lvert A^2\rvert^2-\dfrac{23}{18}S^3.
			\]
			Inserting this into \eqref{integrweitz}, we obtain
			\begin{equation*}
				\int_M |\nabla\operatorname{W}^+|^2 +6c\left(\dfrac{7}{6}S^2
				-2\lvert A^2\rvert^2\right)dV_g=\int_M 15[\operatorname{tr}(A^3)]^2-42\mathrm{tr}(A^6)
				+\dfrac{55}{2}S
				\lvert A^2\rvert^2-\frac{13}{4}S^3dV_g,
			\end{equation*}
			hence deducing \eqref{sharpinequality}.
			Finally, it
			is clear that equality holds in the statement if and only 
			if the self-dual Weyl tensor of $(M,g)$ is parallel: in particular, this implies
			that the eigenvalues of $\operatorname{W}^+$ are constant
			functions on $M$, since $\lvert\operatorname{W^+}\rvert^2$ and $W_{ijkl}^+W_{pqkl}^+W_{ijpq}^+$ are constant \cite{Derdzinski} (recall that 
			$W_{ijkl}^+W_{pqkl}^+W_{ijpq}^+=24\det\operatorname{W}^+$, where $\operatorname{W}^+$ is identified with the self-dual Weyl operator $\mathcal{W}^+$). Hence, by \eqref{bochscalar}, 
			\[
			R\lvert\operatorname{W}^+\rvert^2=6W_{ijkl}^+W_{pqkl}^+W_{ijpq}^+,
			\]
			which allows us to deduce that $R$ is constant: 
			by [Theorem \ref{eigentheorem}, \ref{secondstateWeyl}], we conclude that $M$ is isoparametric. We assume that the eigenvalues of $A$ are simple: since it means that $A$ has four distinct eigenvalues everywhere, we know that $c$ must be equal to $1$ (\cite{cartan, levi1937, segre}). However, since all the eigenvalues have multiplicity $1$ on $M$, $A$ is locally diagonalizable: therefore, around any point $p$, we can write $A_{iit}=(\lambda_i)_t=0$ for every $i,t$. Taking the covariant derivatives of the components defined in \eqref{weylplusminusminimal}, we obtain
			\begin{align*}
				0=W_{1213,t}^+&=\dfrac{1}{4}\left[
				(\lambda_1-\lambda_4)A_{23t}+(\lambda_2-\lambda_3)A_{14t}\right]\\
				0=W_{1214,t}^+&=\dfrac{1}{4}\left[
				(\lambda_1-\lambda_3)A_{24t}+(\lambda_2-\lambda_4)A_{13t}\right]\\
				0=W_{1314,t}^+&=\dfrac{1}{4}\left[
				(\lambda_1-\lambda_2)A_{34t}+(\lambda_3-\lambda_4)A_{12t}\right];
			\end{align*}
			since $\lambda_i\neq\lambda_j$ if $i\neq j$, $A_{iit}=0$ and $A$ is a Codazzi tensor, it is easy to observe that $A_{ijk}=0$ for every $i,j,k$, i.e. $\nabla A\equiv 0$ on $M$. However, this is impossible, due to the results in \cite{cherndocarmokob,lawson}: hence, since Cartan's minimal hypersurfaces do not exist in $N^5(c)$ and $(M^4,g)$ is not locally conformally flat, we conclude that there are exactly two distinct principal curvatures, both with multiplicity two. By minimality and \eqref{ricconst}, we conclude that $(M^4,g)$ must be an Einstein manifold: therefore, since $(M^4,g)$ cannot be totally geodesic, by the results in \cite{Fialkow1938, ryan} we have that $c=1$ and, by Proposition \ref{localisometryprop}, $(M^4,g)$ is locally isometric to $\mathbb{S}^2(1/\sqrt{2})\times\mathbb{S}^2(1/\sqrt{2})$. 
		\end{proof}
			
		\begin{rem}
			In the constant scalar curvature case, by performing analogous computations and exploiting \eqref{simonsasquare}, we get the pointwise equality
			\begin{align}
				|\nabla \operatorname{W}|^2=&4\lvert\nabla A^2\rvert^2+8A_{ij}^2A_{ikt}A_{jkt}+8c\lvert A^2\rvert^2+14cS^2+\\
				&-\left(30[\operatorname{tr}(A^3)]^2
				-84\operatorname{tr}(A^6)+63S\lvert A^2\rvert^2
				-\dfrac{13}{2}S^3\right). \notag
			\end{align}
			Moreover, by \eqref{weylplusminusnormminimal}, it is straightforward to note that, if $c\neq 0$ and $\operatorname{W}\not\equiv 0$, \eqref{sharpinequality} provides a purely extrinsic sharp upper bound for the Weyl functional of a harmonic Weyl minimal hypersurface in $(N^5(c),g_N)$, i.e.
			\[
			c\int_M\lvert\operatorname{W}\rvert^2dV_g\leq \int_M5[\operatorname{tr}(A^3)]^2-14\operatorname{tr}(A^6)+\dfrac{55}{6}S\lvert A^2\rvert^2-\dfrac{13}{12}S^3dV_g.
			\]
		\end{rem}
		The last part of the equality case in Theorem \ref{harmweylinequality} can also be proven by means of the following:
		\begin{proposition} \label{isoparamhhw}
			Any minimal, isoparametric hypersurface with half harmonic Weyl curvature of a space form $(N^5(c),g_N)$ is Einstein or locally conformally flat. 
		\end{proposition}
		\begin{proof}
			An easy computation based on \eqref{divergplusmin} shows that, at every point, the components of $\delta\operatorname{W}^+$ can be written as
			\begin{align} \label{divergweylplus}
				(\delta\operatorname{W}^{\pm})_{112}&=\dfrac{1}{4}\left[
				(\lambda_1-\lambda_2)A_{112}\pm(\lambda_3-\lambda_4)A_{134}\right], \quad
				(\delta\operatorname{W}^{\pm})_{113}=\dfrac{1}{4}\left[
				(\lambda_1-\lambda_3)A_{113}\pm(\lambda_4-\lambda_2)A_{142}\right]\\
				(\delta\operatorname{W}^{\pm})_{114}&=\dfrac{1}{4}\left[
				(\lambda_1-\lambda_4)A_{114}\pm(\lambda_2-\lambda_3)A_{123}\right]\notag, \quad
				(\delta\operatorname{W}^{\pm})_{212}=\dfrac{1}{4}\left[
				(\lambda_1-\lambda_2)A_{212}\pm(\lambda_3-\lambda_4)A_{234}\right]\notag\\
				(\delta\operatorname{W}^{\pm})_{213}&=\dfrac{1}{4}\left[
				(\lambda_1-\lambda_3)A_{213}\pm(\lambda_4-\lambda_2)A_{242}\right]\notag, \quad
				(\delta\operatorname{W}^{\pm})_{214}=\dfrac{1}{4}\left[
				(\lambda_1-\lambda_4)A_{214}\pm(\lambda_2-\lambda_3)A_{223}\right]\notag\\
				(\delta\operatorname{W}^{\pm})_{312}&=\dfrac{1}{4}\left[
				(\lambda_1-\lambda_2)A_{312}\pm(\lambda_3-\lambda_4)A_{334}\right]\notag, \quad
				(\delta\operatorname{W}^{\pm})_{313}=\dfrac{1}{4}\left[
				(\lambda_1-\lambda_3)A_{313}\pm(\lambda_4-\lambda_2)A_{342}\right]\notag\\
				(\delta\operatorname{W}^{\pm})_{314}&=\dfrac{1}{4}\left[
				(\lambda_1-\lambda_4)A_{314}\pm(\lambda_2-\lambda_3)A_{323}\right]\notag, \quad
				(\delta\operatorname{W}^{\pm})_{412}=\dfrac{1}{4}\left[
				(\lambda_1-\lambda_2)A_{412}\pm(\lambda_3-\lambda_4)A_{434}\right]\notag\\
				(\delta\operatorname{W}^{\pm})_{413}&=\dfrac{1}{4}\left[
				(\lambda_1-\lambda_3)A_{413}\pm(\lambda_4-\lambda_2)A_{442}\right], \quad 
				(\delta\operatorname{W}^{\pm})_{414}=\dfrac{1}{4}\left[
				(\lambda_1-\lambda_4)A_{414}\pm(\lambda_2-\lambda_3)A_{423}\right]\notag.
			\end{align}
			Since $(M^4,g)$ is isoparametric, it is easy to see that 
			\[
			\lvert\delta\operatorname{W}^+\rvert=\lvert\delta\operatorname{W}^-\rvert,
			\]
			which, by the fact that $\nabla\operatorname{W}$ splits orthogonally into $\nabla\operatorname{W}^++\nabla\operatorname{W}^-$, implies that $(M^4,g)$ is half harmonic Weyl if and only if $\delta\operatorname{W}\equiv 0$. 
			Since $\delta\operatorname{W}=0$ and $R$ is constant, a result due to Umehara \cite{umehara} allows us to conclude that $m\leq 2$, i.e. $(M^4,g)$ is either Einstein or locally conformally flat, which proves the claim. 
		\end{proof}
		
		\begin{rem}
			We point out that, if $(M^4,g)$ a closed, minimally immersed hypersurface in $N^5(1)$ with constant scalar curvature, then, if $\delta\operatorname{W}^+=0$, there exists at least one point where $m\leq 3$. Indeed, if we assume that $m=4$ everywhere, by the fact that $A$ can be locally diagonalizable around every point we obtain that $A_{iit}=(\lambda_i)_t$: by \eqref{divergweylplus}, it is immediate to obtain the following local conditions on $\nabla A$
			\begin{align*}
				A_{123}&=\dfrac{\lambda_2-\lambda_4}{\lambda_1-\lambda_3}(\lambda_2)_4=\dfrac{\lambda_4-\lambda_3}{\lambda_1-\lambda_2}(\lambda_3)_4=\dfrac{\lambda_4-\lambda_1}{\lambda_2-\lambda_3}(\lambda_1)_4;\\
				A_{124}&=\dfrac{\lambda_3-\lambda_1}{\lambda_4-\lambda_2}(\lambda_1)_3=\dfrac{\lambda_3-\lambda_2}{\lambda_1-\lambda_4}(\lambda_2)_3=\dfrac{\lambda_4-\lambda_3}{\lambda_1-\lambda_2}(\lambda_4)_3;\\
				A_{134}&=\dfrac{\lambda_2-\lambda_1}{\lambda_3-\lambda_4}(\lambda_1)_2=\dfrac{\lambda_3-\lambda_2}{\lambda_1-\lambda_4}(\lambda_3)_2=\dfrac{\lambda_2-\lambda_4}{\lambda_1-\lambda_3}(\lambda_4)_2;\\
				A_{234}&=\dfrac{\lambda_2-\lambda_1}{\lambda_3-\lambda_4}(\lambda_2)_1=\dfrac{\lambda_3-\lambda_1}{\lambda_4-\lambda_2}(\lambda_3)_1=\dfrac{\lambda_4-\lambda_1}{\lambda_2-\lambda_3}(\lambda_4)_1.
			\end{align*}
			Since $M$ is closed, we can consider, for instance, a point $p\in M$ such that $\nabla\lambda_1=0$: by the previous equations and the fact that $M$ is minimal, at that point we have $\nabla A=0$. Since the hypersurface cannot be totally geodesic by hypothesis, Simons' identity \eqref{simonsid} immediately implies that $S\equiv n$, which contradicts Proposition \ref{localisometryprop}. 
			
			Also, we point out that, if $1<m\leq 3$ everywhere, then, by Theorem \ref{eigentheorem} and a well-known result by Derdzi\'{n}ski \cite{Derdzinski}, around every point $p\in M$ such that $\lvert \operatorname{W}^+\rvert\neq 0$ there exists an open neighborhood in which the metric $g$ is conformally K\"{a}hler. 
		\end{rem}
		Finally, we want to analyze a weaker condition, which is satisfied both by Einstein and half conformally flat metrics: in the closed, four-dimensional case, the critical points of the so-called \emph{Weyl functional}
		\[
		g\longmapsto \int_M\lvert\operatorname{W}\rvert^2dV_g
		\]
		are the ones who satisfy the Euler--Lagrange equations
		\begin{equation} \label{bachtensor}
			0=W_{ikjl,lk}+\dfrac{1}{2}R_{kl}W_{ikjl}=:B_{ij}.
		\end{equation}
		The functions $B_{ij}$ are the local components of the \emph{Bach tensor}
		$\operatorname{B}$ \cite{bach}, which enjoys a plenty of interesting properties 
		in four dimensions: for instance, $\operatorname{B}$ is 
		symmetric, trace-free in all dimensions, but in the four-dimensional
		case it is also divergence-free and \emph{conformally covariant}, since,
		under a pointwise conformal change of metric
		$\tilde{g}=e^{2u}g$, it transforms as
		\cite{CatinoMastroliaBook}
		\[
		e^{2u}\tilde{\operatorname{B}}=\operatorname{B}.
		\]
		Indeed, the Weyl functional is conformally invariant, hence its
		Euler--Lagrange equations are conformally invariant conditions:
		this implies that, if a metric $g$ is \emph{Bach-flat}, i.e. its
		Bach tensor identically vanishes, every Riemannian metric in 
		the conformal class $[g]$ is Bach-flat. 
		
		In the context of minimal hypersurfaces isometrically embedded into
		space forms, the Bach-flat condition allows us to prove some
		Bochner-type formulas, which lead to integral equalities and sharp
		bounds on the second fundamental form. 
		By \eqref{bachtensor}, it is possible to show that
		(see \cite{CatinoMastroliaBook} for higher dimensional cases):
		\begin{equation} \label{bachtensor2}
			2B_{ij}=\Delta R_{ij}-\dfrac{1}{3}R_{,ij}+
			2R_{kt}R_{ikjt}-\dfrac{2}{3}RR_{ij}-\dfrac{1}{6}\Delta R\delta_{ij}
			-\dfrac{1}{2}\left(\lvert\operatorname{Ric}\rvert^2
			-\dfrac{R^2}{3}\right)\delta_{ij}, \notag
		\end{equation}
		where $R_{,ij}$ are the components of $\operatorname{Hess}R$. 
		We recall that, in the proof of Theorem \ref{sharpquadrbound}, we also derived a Simons' type identity,
		which, for four-dimensional minimal hypersurfaces, becomes
		\begin{equation} \label{laplasquare}
			\Delta A_{ij}^2=2(4c-S)A_{ij}^2+2A_{ikt}A_{jkt}.
		\end{equation}
		By \eqref{ricconst}, \eqref{scalconst}, \eqref{simonsidlocal}, \eqref{simonsid} and \eqref{laplasquare}, we obtain
		\begin{align} \label{bachhyper}
			2B_{ij}&=2A_{ij}^4-2A_{ij}\operatorname{tr}(A^3)-4cA_{ij}^2+\dfrac{4}{3}SA_{ij}^2+\dfrac{1}{3}S_{ij}-2A_{ikt}A_{jkt}+\\
			&+\left(\dfrac{1}{3}\lvert\nabla A\rvert^2+\dfrac{1}{3}cS-\dfrac{1}{6}S^2-\dfrac{1}{2}\lvert A^2\rvert^2\right)\delta_{ij}, \notag
		\end{align}
		where $S_{ij}$ are the local components of $\operatorname{Hess} S$.
		Now, we are ready to show a rigidity integral result for Bach-flat hypersurfaces (Theorem \ref{bachident}).
		\begin{proof}[Proof of Theorem \ref{bachident}] 
			By hypothesis, $g$ is a Bach-flat metric,
			hence, the left-hand side of 
			\eqref{bachhyper} vanishes. 
			Now, we contract \eqref{bachhyper} with $A_{ij}$ to obtain
			\[
			0=-A_{ij}\Delta A_{ij}^2+
			2\operatorname{tr}(A^5)+4c\operatorname{tr}(A^3)-
			\dfrac{8}{3}S\operatorname{tr}(A^3);
			\]
			using \eqref{laplasquare} in the previous equality, we get
			\eqref{firstbochnerbach}. 
			
			Now, using integration by parts and renaming the indices, we obtain
			\begin{align*}
				\int_MA_{ij}A_{ikl}A_{jkl}dV_g&=-
				\int_MA_{ijl}A_{ik}A_{jkl}dV_g-\int_MA_{ij}A_{ik}\Delta A_{jk}dV_g=\\
				&=-\int_MA_{ij}A_{ikl}A_{jkl}dV_g-\left(4c-S\right)
				\int_M\operatorname{tr}(A^3)dV_g,
			\end{align*}
			that is,
			\[
			\int_MA_{ij}A_{ikl}A_{jkl}dV_g=-\dfrac{1}{2}\left(4c-S\right)
			\int_M\operatorname{tr}(A^3)dV_g;
			\]
			furthermore, since $A$ is a Codazzi trace-free tensor, we get
			\[
			\int_MA_{ij}S_{ij}=-\int_MA_{ijj}S_i=0.
			\]
			Integrating \eqref{firstbochnerbach} and using these facts,
			we immediately obtain the integral identity
			\[
			\int_M\operatorname{tr}(A^5)dV_g=\dfrac{5}{6}\int_M S\operatorname{tr}(A^3)dV_g;
			\]
			hence, the validity of \eqref{bachintegral1} and the characterization of the equality case follow directly from  \eqref{lagrange}. 
			
			Now, let us contract \eqref{bachhyper} with $A_{ij}^2$ in order
			to have
			\begin{align*}
				0&=2\operatorname{tr}(A^6)
				-2\left[\operatorname{tr}(A^3)\right]^2-4c\lvert A^2\rvert^2+
				\dfrac{5}{6}S\lvert A^2\rvert^2-
				\dfrac{1}{6}S^3+\dfrac{1}{3}cS^2+\\
				&+\dfrac{1}{3}\lvert\nabla A\rvert^2+\dfrac{1}{3}A_{ij}^2S_{ij}-2A_{ij}^2A_{ikt}A_{jkt};
			\end{align*}
			by a straightforward application of \eqref{simonsid} and \eqref{simonsasquare}, we obtain \eqref{secondbochnerbach}.
			
			Now, we can integrate \eqref{secondbochnerbach}: first, a simple integration by parts, combined with \eqref{simons2} and the fact that 
			$A$ is Codazzi, shows that
			\[
			\int_MA_{ij}^2S_{ij}\,dV_g=-\int_M
			(A_{it}A_{tj})_jS_i\,dV_g=-\int_M
			S_iA_{ijt}A_{jt}\,dV_g=
			\int_M\left(SA_{ijti}A_{jt}+S\lvert\nabla A\rvert^2\right)dV_g.
			\]
			Now, since $M$ is minimal, we can apply a well-known commutation formula for the Hessian of $(0,2)$-tensor fields
			(see e.g. \cite{CatinoMastroliaBook}) in order to obtain, by \eqref{riemconst} and \eqref{ricconst},
			\begin{align*}
				A_{ijti}&=A_{ijit}+A_{pj}R_{pt}+A_{ip}R_{pjti}=\\
				&=A_{pj}(3c\delta_{pt}-A_{pt}^2)+
				A_{ip}[c(\delta_{pt}\delta_{ji}-\delta_{pi}\delta_{jt}+A_{pt}A_{ji}-A_{pi}A_{jt}]=\\
				&=(4c-S)A_{jt};
			\end{align*}
			this, by \eqref{simons2}, immediately implies that
			\[
			\int_M A_{ij}^2S_{ij}=\int_M \left[S\lvert\nabla A\rvert^2+S^2(4c-S)\right]dV_g=
			\int_M\left[\dfrac{1}{4}\Delta S^2-\dfrac{1}{2}\lvert\nabla S\rvert^2\right]dV_g=-\dfrac{1}{2}\int_M\lvert\nabla S\rvert^2dV_g.
			\]
			Now, we can exploit
			the first inequality in \eqref{ineq} and \eqref{lagrange} to deduce \eqref{bachintegral2}. 
			
			Finally, equality holds in \eqref{bachintegral2} if and only if
			equalities hold in the right-hand side of \eqref{ineq} and in \eqref{lagrange} at every point $p\in M$, which means that
			$(M,g)$ is a locally conformally flat manifold, by Proposition \ref{nishi}.
		\end{proof}
		\begin{rem}
			We highlight the fact that \eqref{secondbochnerbach} is the
			extrinsic version of the Bochner--Weitzenb\"{o}ck formula
			for Bach-flat metrics proven in \cite{Chang-gursky-yang-annals}.
		\end{rem}
		Now, we can prove an integral equality in the locally conformally
		flat case, derived from Theorem \ref{bachident}.
		\begin{proof}[Proof of Corollary \ref{corbachlcf}]
			The claim follows immediately by the fact that equality holds in \eqref{bachintegral2} and
			that, by Proposition \ref{nishi}, at every point there are three equal principal
			curvatures: hence, if they are $-3\lambda$, $\lambda$, $\lambda$, $\lambda$, it is 
			immediate to see that
			\[
			\lambda=\pm\sqrt{\dfrac{S}{12}} \quad \Longrightarrow \quad \operatorname{tr}(A^6)=\dfrac{61}{144}S^3.
			\]
			Inserting this into \eqref{bachintegral2} and using the equality in the right-hand side of \eqref{ineq},
			the claim is proven. 
		\end{proof}
		
		We finally point out that, using \eqref{ineq} and \eqref{lagrange}, by Theorem \ref{harmweylinequality} and Theorem \ref{bachident} we can obtain integral rigidity results for half harmonic Weyl and Bach-flat minimal hypersurfaces (Corollary \ref{rigresultintegral}): 
		in particular, we prove a strict inequality for the second fundamental form
		in the case of half harmonic Weyl metrics (\eqref{strictharmweyl}) and a characterization of totally geodesic hypersurfaces in positive space forms,
		under the Bach-flat hypothesis (\eqref{totgeodbach}).
		
		
		\begin{proof}[Proof of Corollary \ref{rigresultintegral}]
			\begin{enumerate}
				\item We assume that $(M^4,g)$ is not locally conformally flat: by \eqref{sharpinequality} and 
				\eqref{lagrange} we get 
				\[
				c\int_M\left(\dfrac{7}{3}S^2-4\lvert A^2\rvert^2\right)dV_g \leq \int_M\left(\dfrac{427}{72}S^3-14\operatorname{tr}(A^6)\right)dV_g.
				\]
				In order to prove \eqref{strictharmweyl}, 
				we need to show that the previous inequality
				is strict, but this is trivial: indeed, if 
				equality holds, it has to hold in 
				\eqref{lagrange} at every $p\in M$, which means
				that at least 
				three principal curvatures must be equal at every $p\in M$. Therefore, Proposition \ref{nishi} would imply that $(M^4,g)$ is locally conformally
				flat, contradicting our assumption. 
				\item The validity of \eqref{totgeodbach} is deduced from the fact that, since $c=1$, by the left-hand side of \eqref{ineq} the last term in the right-hand side of \eqref{bachintegral2} is non-positive: hence the inequality is proven. Finally, by \eqref{ineq} and Theorem \ref{bachident}, equality holds if and only if $(M^4,g)$ is Einstein and locally conformally flat, which means that $(M^4,g)$ is a totally geodesic positive space form.
			\end{enumerate}
		\end{proof}
		\bibliographystyle{abbrv}
		\bibliography{bibliography}
		
		
		
		

		\end{document}